\documentclass[a4paper,10pt]{amsart}
\usepackage{mathrsfs}
\usepackage{amsfonts}
\usepackage{txfonts}
\usepackage[arrow,matrix]{xy}
\usepackage{amsmath,amssymb,amscd,bbm,amsthm,mathrsfs,dsfont,color,hyperref}
\usepackage{changes}

\theoremstyle{plain} \textwidth=36pc \textheight=51pc

\newtheorem{theorem}{Theorem}[section]
\newtheorem{lemma}[theorem]{Lemma}
\newtheorem{exa}[theorem]{Example}
\newtheorem{thm}[theorem]{Theorem}
\newtheorem{prop}[theorem]{Proposition}
\newtheorem{corollary}[theorem]{Corollary}
\theoremstyle{definition}
\newtheorem{defn}[theorem]{Definition}
\newtheorem{remark}[theorem]{Remark}
\numberwithin{equation}{section}

\DeclareMathOperator{\End}{End}
\DeclareMathOperator{\Hom}{Hom}
 \DeclareMathOperator{\im}{im}

\begin{document}
\title[Universal enveloping algebras of DG Poisson algebras]{Universal enveloping algebras of differential \\graded Poisson algebras}

\author{Jiafeng L\"u}
\address{L\"u: Department of Mathematics, Zhejiang Normal University, Jinhua, Zhejiang 321004, P.R. China}
\email{jiafenglv@zjnu.edu.cn, jiafenglv@gmail.com}

\author{Xingting Wang}
\address{Wang: Department of Mathematics,
University of Washington, Seattle, Washington 98195, USA}
\email{xingting@uw.edu}

\author{Guangbin Zhuang}
\address{Zhuang: Department of Mathematics,
University of Southern California, Los Angeles 90089-2532, USA}

\email{gzhuang@usc.edu}

\begin{abstract}
In this paper, we introduce the notion of {\it differential graded Poisson algebra} and study its {\it universal enveloping algebra}. From any differential graded Poisson algebra $A$, we construct two isomorphic differential graded algebras: $A^e$ and $A^E$.  It is proved that the category of differential graded Poisson modules over $A$ is isomorphic to the category of differential graded modules over $A^e$, and $A^e$ is the unique universal enveloping algebra of $A$ up to isomorphisms. As applications of the universal property of $A^e$, we prove that $(A^e)^{op}\cong (A^{op})^e$ and $(A\otimes_{\Bbbk}B)^e\cong A^e\otimes_{\Bbbk}B^e$ as differential graded algebras. As consequences, we obtain that ``$e$'' is a monoidal functor and establish links among the universal enveloping algebras of differential graded Poisson algebras, differential graded Lie algebras and associative algebras.
\end{abstract}
\maketitle
\section{Introduction}
The notion of {\it Poisson algebra} originally arises from Poisson geometry, which can be considered as a ``combination'' of commutative algebras and Lie algebras. Recently, many interesting generalizations on Poisson algebras have been obtained in both commutative and noncommutative settings: Poisson orders (\cite{BG}), graded Poisson algebras (\cite{C}, \cite{CFL}), noncommutative Leibniz-Poisson algebras (\cite{CD}), left-right noncommutative Poisson algebras (\cite{CDL}), Poisson PI algebras (\cite{MPR}), double Poisson algebras (\cite{Vd}), noncommutative Poisson algebras (\cite{Xup}), Novikov-Poisson algebras (\cite{X}) and Quiver Poisson algebras (\cite{YYZ}), etc. It should be noted that, by using the ``deformation quantization'' theory (\cite{EK}, \cite{H}, \cite{K}), Poisson structures can be used as a tool to study noncommutative algebras. For example, by considering the Poisson structure on the center of Weyl algebras, Belov and Kontsevich proved that the Jacobian Conjecture is stably equivalent to the Dixmier Conjecture (\cite{BK}). One of the interesting aspects of Poisson algebras is the notion of {\it Poisson universal enveloping algebra}, which was first introduced by Oh in 1999 (\cite{Oh1}), and later Oh, Park and Shin studied the PBW-basis for Poisson universal enveloping algebras (\cite{OPS}). Most recently, many people show their interests in Poisson universal enveloping algebras: Umirbaev studied the universal enveloping algebras and universal derivations of Poisson algebras (\cite{U}), the authors of the present paper studied the universal enveloping algebras of Hopf Poisson algebras and Poisson Ore-extensions (\cite{LWZ1}, \cite{LWZ2}), and also studied the Artin-Schelter regularity of Poisson universal enveloping algebras (\cite{LWZ4}); Yang, Yao and Ye studied the universal enveloping algebras of noncommutative Poisson algebras (\cite{YYY}), etc.

We know that, differential graded algebras arise naturally in algebra, representation theory and algebraic topology. For instance, Koszul complexes,
endomorphism algebras, fibres of ring homomorphisms, singular chain and cochain algebras of topological spaces, and bar resolutions all admit natural differential graded algebra structures. Moreover, differential graded algebras play an important role in both commutative and noncommutative algebras. For example, differential graded commutative algebra provides many powerful techniques for proving theorems about modules over commutative rings (\cite{BS}). For the brief introduction, basic properties and some nice homological properties of differential graded algebras, one can see (\cite{AFH}) and (\cite{FIJ}-\cite{FJ2}) for the further details.

Motivated by the notion of {\it differential graded Lie algebra} (\cite{GM}) and the fact that Lie algebras is a special class of Poisson algebras, one can ask the following natural question:
\begin{itemize}
  \item What will happen if we combine the structures of differential graded algebras and graded Poisson algebras?
\end{itemize}

As an attempt to answer the above question, we define the {\it differential graded Poisson algebra} in this paper, which is a differential graded algebra with a compatible Poisson structure (Definition \ref{dgpa}). The present paper is the first paper on our project of ``differential graded Poisson algebras and deriving differential graded Poisson categories''. The main aim of this paper is to give the notion of a {\it differential graded Poisson algebra} and some basic definitions and properties of such algebras. In particular, in order to study the category of left differential graded Poisson modules over differential graded Poisson algebras, we prove the following results:

{\noindent\bf Proposition A} {\it (Lemma \ref{lem1}, Proposition \ref{dg}, Theorems \ref{same}) Let $A$ be any differential graded Poisson algebra. Then $A^e\cong A^E$ as differential graded algebras. Here the constructions of $A^e$ and $A^E$ are given in Subsections 3.1.1 and 3.1.2}

Because we have $A^e\cong A^E$ for any differential graded Poisson algebra $A$, we choose to consider the case of $A^e$ in the rest of the paper.

{\noindent\bf Theorem B} {\it (Theorems \ref{equivalence} and \ref{universalproperty}) Let $(A, d, \{,\})$ be a differential graded Poisson algebra. Then we have the following statements.
\begin{enumerate}
  \item We have $\mathbf{DGP}(A)\cong\mathbf{DG}(A^e)$, where $\mathbf{DGP}(A)$ denotes the category of left differential graded Poisson modules over $A$, and $\mathbf{DG}(A^e)$ denotes the category of left differential graded modules over $A^e$.
  \item The differential graded algebra $A^e$ is the universal enveloping algebra of $A$. That is, $(A^e, \partial)$ together with the differential graded algebra map $m: (A,d)\to (A^e,\partial)$ and the differential graded Lie algebra map $h:(A,\{,\},d)\to (A^e, [,], \partial)$ satisfy the following ``universal property'': for any differential graded algebra $(D,\delta)$ with a differential graded algebra map $f: (A,d)\to (D,\delta)$ and a differential graded Lie algebra map $g:(A,\{,\},d)\to(D,[,],\delta)$ satisfying
\begin{align*}
f(\{a,b\})=g(a)f(b)-(-1)^{|a||b|}f(b)g(a),\\
g(ab)=f(a)g(b)+(-1)^{|a||b|}f(b)g(a),
\end{align*}
for any homogeneous elements $a,b\in A$, then there exists a unique differential $\mathbb{Z}$-graded algebra map $\phi: (A^{e},\partial)\to(D,\delta)$, making the diagram
\[
\xymatrix{
A\ar[rr]^-{m}_-{h}\ar[dr]^-{f}_-{g} && A^e\ar@{-->}[dl]^-{\exists ! \phi}\\
& D&
}
\]
``bi-commute'', i.e., $\phi m=f$ and $\phi h=g$.
\end{enumerate}}

Note that the ``universal property'' of $A^e$ is different from the classical universal property, there does not exist an obvious adjoint functor (see Remark \ref{u}). As applications of the ``universal property'' of $A^e$, we prove the following result:

{\noindent\bf Theorem C} {\it (Proposition \ref{tensor}, Corollaries \ref{e} and \ref{mono}, Theorems \ref{thop} and \ref{thotimes}) Let $\mathbf{DGPA}$ be the category of differential graded Poisson algebras, and $\mathbf{DGA}$ be the category of differential graded algebras. Then
\begin{enumerate}
  \item $e:\mathbf{DGPA}\to \mathbf{DGA}$ is a covariant functor between these two symmetric monoidal categories;
  \item $(A^{op})^e\cong(A^e)^{op}$ for any $A\in \mathbf{DGPA}$;
  \item $(A\otimes B)^e\cong A^e\otimes B^e$ for any $A,B\in \mathbf{DGPA}$;
  \item $\Bbbk^e\cong\Bbbk$.
\end{enumerate}
As a consequence, we have that ``$e$'' is a monoidal functor.}

It should be noted that [(2), (3), Theorem C] is helpful to
study the ``Hochschild (co)homology theory'' in the category of differential graded Poisson bi-modules (Remark \ref{4.13}).

Note that by now we have three meanings of the phrase: ``universal enveloping algebra''---universal enveloping algebra of a Poisson algebra $A$, universal enveloping algebra of a Lie algebra $L$ and universal enveloping algebra of an associative algebra $R$, we usually denote them by $A^e$, $\mathbf{U}(L)$ and $R^{\mathcal{E}}$, respectively. Fortunately, we can build some links among them in the case of differential graded Poisson algebras:

{\noindent\bf Corollary D} {\it (Propositions \ref{p1} and \ref{sem}, Theorem \ref{semi} and Corollary \ref{coop}) The following statements are true.
\begin{enumerate}
  \item Let $(L,[,]_L,d)$ be a differential graded Lie algebra and $\mathbb{S}L$ be its graded symmetric algebra. Then $\mathbb{S}L$ is a differential graded Poisson algebra and  $(\mathbb{S}L)^e\cong\mathbf{U}(L\rtimes L)$ as differential graded algebras, where $L\rtimes L$ denotes the graded semidirect product of $L$, which is also a differential graded Lie algebra.
  \item Let $A\in \mathbf{DGPA}$. Then $(A^{\mathcal{E}})^e\cong(A^e)^{\mathcal{E}}$, where $A^{\mathcal{E}}:=A\otimes A^{op}$.
\end{enumerate}}

The paper is organized as follows:
In Section 2, we give the definitions, examples and basic properties of differential graded Poisson algebras and differential graded Poisson modules. In particular, we prove that $\mathbf{DGPA}$ is a symmetric monoidal category with a left and right identity $\Bbbk$. Section 3 is devoted to studying the universal enveloping algebra of a differential graded Poisson algebra. For any $A\in\mathbf{DGPA}$, Firstly we construct two differential graded algebras $A^e$ and $A^E$ and prove that $A^e\cong A^E$ as differential graded algebras. Then we choose to consider the case of $A^e$ and prove that $\mathbf{DGP}(A)\cong\mathbf{DG}(A^e)$, $A^e$ is the unique universal enveloping algebra of $A$ up to isomorphisms, $e:\mathbf{DGPA}\to \mathbf{DGA}$ is a covariant functor between these two symmetric monoidal categories and $(\mathbb{S}L)^e\cong \mathbf{U}(L\rtimes L)$ as differential graded algebras.
The last section provides some applications of the ``universal property'' of universal enveloping algebras of differential graded Poisson algebras. We prove that $(A^e)^{op}\cong (A^{op})^e$ and $(A\otimes_{\Bbbk}B)^e\cong A^e\otimes_{\Bbbk}B^e$ as differential graded algebras. As consequences, we have that $e:\mathbf{DGPA}\to\mathbf{DGA}$ in fact is a monoidal functor and $(A^{\mathcal{E}})^e\cong (A^e)^{\mathcal{E}}$ as differential graded algebras.

Finally we would like to point out that although the degree of the Poisson bracket is zero in our definition of differential graded Poisson algebras, the results obtained in this paper are also true for differential graded Poisson algebras of degree $n$ with some expected signs, where $n\in \mathbb{Z}$ is the degree of the Poisson bracket.

Throughout, $\mathbb{Z}$ denotes the set of integers, $\Bbbk$ denotes a base field and everything is over $\Bbbk$ unless otherwise stated, all (graded) algebras are assumed to have an identity and all (graded) modules are assumed to be unitary.
\subsection*{Acknowledgments} The authors first want to give their sincere gratitudes to James Zhang for introducing them this project. They also want to thank Yanhong Bao, Jiwei He, Xuefeng Mao, Cris Negron and James Zhang for many valuable discussions and suggestions on this paper. The first author is partially supported by National Natural Science Foundation of China [11001245, 11271335 and 11101288] and the second author is partially supported by U.S. National Science Foundation [DMS0855743].

\medskip
\section{Preliminaries}
In this section, we mainly give some basic definitions, examples and properties of differential graded Poisson algebras and differential graded Poisson modules. We usually omit the proofs since most of them are routine checks.

\subsection{Differential graded Poisson algebras}
In this subsection, we will introduce the notion of a {\it differential graded Poisson algebra} and give some examples and basic properties of such algebras.
\begin{defn}\label{dga}
Let $(A,\cdot)$ be a $\mathbb{Z}$-graded algebra with a $\Bbbk$-linear homogeneous map $d: A\to A$ of degree $1$. If $d^2=0$, and
$$d(a\cdot b)=d(a)\cdot  b+(-1)^{|a|}a\cdot d(b),$$
for any homogeneous elements $a,\;b\in A$, then $A$ is called a {\it differential $\mathbb{Z}$-graded algebra}, which is usually denoted by $(A,\cdot, d)$, or simply by $(A,d)$ or $A$ if no confusions arise. Here $d$ is called the {\it differential} of $A$ and $|x|$ denotes the degree of any homogeneous element $x$.

Let $A,B$ be two differential graded algebras and $f:A\to B$ be a graded algebra map of degree zero. Then $f$ is called a {\it differential graded algebra map} if $f$ commutes with the differentials.

Denote by $\mathbf{DGA}$ the category of differential graded algebras whose morphism space consists of differential graded algebra maps.
\end{defn}

\begin{defn}\label{gla}
Let $A$ be a $\mathbb{Z}$-graded $\Bbbk$-vector space. If there is a $\Bbbk$-linear map
$$\{,\}: A\otimes A\to A$$ of degree 0 such that
\begin{itemize}
  \item[(i)] (graded antisymmetry): $\{a,b\}=-(-1)^{|a||b|}\{b,a\}$;
  \item[(ii)] (graded Jacobi identity): $\{a,\{b,c\}\}=\{\{a,b\},c\}+(-1)^{|a||b|}\{b,\{a,c\}\}$,
  \end{itemize}
for any homogeneous elements $a,b,c\in A$, then $(A,\{,\})$ is called a {\it graded Lie algebra}. If in addition, there is a $\Bbbk$-linear map $d: A\to A$ of degree 1 such that  $d^2=0$ and
\begin{itemize}
  \item [(iii)] (graded Leibniz rule for bracket): $d(\{a,b\})=\{d(a),b\}+(-1)^{|a|}\{a,d(b)\}$,
  \end{itemize}
 for any homogeneous elements $a,b\in A$, then $A$ is called a {\it differential graded Lie algebra}, which is usually denoted by $(A,\{,\},d)$, or simply by $A$ if no confusions arise.
\end{defn}

\begin{defn}\label{dgpa}
Let $(A,\cdot)$ be a $\mathbb{Z}$-graded $\Bbbk$-algebra. If there is a $\Bbbk$-linear map
$$\{,\}: A\otimes A\to A$$ of degree 0 such that
\begin{itemize}
  \item[(i)] $(A,\{,\})$ is a graded Lie algebra;
  \item [(ii)] (graded commutativity): $a\cdot b=(-1)^{|a||b|}b\cdot a$; \item[(iii)] (biderivation property): $\{a,b\cdot c\}=\{a,b\}\cdot c+(-1)^{|a||b|}b\cdot\{a,c\}$,
  \end{itemize}
for any homogeneous elements $a,b,c\in A$. Then $A$ is called a {\it graded Poisson algebra}. If in addition, there is a $\Bbbk$-linear homogeneous map $d: A\to A$ of degree 1 such that  $d^2=0$ and
\begin{itemize}
  \item [(iv)] (graded Leibniz rule for bracket): $d(\{a,b\})=\{d(a),b\}+(-1)^{|a|}\{a,d(b)\}$;
  \item [(v)] (graded Leibniz rule for product): $d(a\cdot b)=d(a)\cdot  b+(-1)^{|a|}a\cdot d(b)$,
  \end{itemize}
for any homogeneous elements $a,b\in A$, then $A$ is called a {\it differential graded Poisson algebra}, which is usually denoted by $(A,\cdot, \{,\},d)$, or simply by $(A, \{,\},d)$ or $A$ if no confusions arise.
\end{defn}

Now we give some examples of differential graded Poisson algebras.\begin{exa}\label{eg1}
{\rm The following are some trivial examples of differential graded Poisson algebras.
\begin{enumerate}
  \item Let $(A, \cdot,\{,\},d)$ be any differential graded Poisson algebra. Then $(A^0,\cdot)$ is a usual commutative algebra; $(A^0, \{,\})$ is a usual Lie algebra  and $(A^0,\cdot, \{,\})$ is a usual Poisson algebra. Conversely, any commutative algebra can be considered as a differential graded Poisson algebra concentrated in degree 0 with trivial differential and trivial Poisson bracket; any Lie algebra can be considered as a differential graded Poisson algebra concentrated in degree 0 with trivial differential and trivial product; and any Poisson algebra can be considered as a differential graded Poisson algebra concentrated in degree 0 with trivial differential.
  \item Differential graded commutative algebras are differential graded Poisson algebras with trivial Poisson bracket.
  \item Graded Poisson algebras are differential graded Poisson algebras with trivial differential.
  \item Differential graded Lie algebras are differential graded Poisson algebras with trivial product.
\end{enumerate}}
\end{exa}

\begin{exa}\label{con}
{\rm Let \[A=\frac{\Bbbk\left<x_1,x_2\right>}{(x_1^2,x_1x_2-x_2x_1)}\]with $|x_1|=1, |x_2|=2$. Then $A$ is a differential graded Poisson algebra with
\begin{align*}
d(x_1)=\lambda x_2, d(x_2)=\mu x_1x_2,\\
\{x_1,x_2\}=px_1x_2,\end{align*}
where $\lambda,\mu, p\in\Bbbk$ and $\lambda\mu=0$.}
\end{exa}

\begin{exa}\label{eg2}
{\rm Let $(V,d)$ be a differential $\mathbb{Z}$-graded $\Bbbk$-vector space and $W:=\Hom^{\ast}_{\Bbbk}(V,V)$ be the total Hom complex.  Set
$$A=\frac{T(W)}{(f\otimes g-(-1)^{|f||g|}g\otimes f)},$$ where $T(W)$ is the tensor algebra over $\Bbbk$, and $f,g\in W$ are any homogeneous elements. Then $(A,\otimes,\{,\},\partial)$ is a differential graded Poisson algebra where
\begin{align*}
\partial(f):=d\circ f-(-1)^{|f|}f\circ d, \\
\{f, g\}:=f\circ g-(-1)^{|f||g|}g\circ f.\end{align*}}
\end{exa}

\begin{exa}\label{eg3}
{\rm Let $E$ be a holomorphic vector bundle on a complex variety $M$, and $\mathcal{A}^{p,q}(E)$ be the sheaf of $C^{\infty}$ $(p,q)$-differential forms with values in $E$.

Set
$$V=\bigoplus_{q}\Gamma(M,\mathcal{A}^{0,q}(\End(E)))[-q]$$and
$$A=\frac{T(V)}{(u\otimes v-(-1)^{|u||v|}v\otimes u)}.$$
Here $u,v\in V$ are any homogeneous elements. Then $(A,\otimes,\{,\},d)$ is a differential graded Poisson algebra via
\begin{align*}
d(\phi e):=(\overline{\partial}\phi)e, \\
\{\phi e, \psi g\}:=\phi\wedge\psi[e,g],
\end{align*}
where $\overline{\partial}$ is the Dolbeault differential, $e,g$ are local holomorphic sections of $\End(E)$ and $\phi, \psi$ are differential forms.}\end{exa}

Note that Examples \ref{eg2} and \ref{eg3} are the graded symmetric algebras of differential graded Lie algebras. In fact, we have a general result:
\begin{prop}\label{p1}
Let $(L, d, [,])$ be a differential $\mathbb{Z}$-graded Lie algebra. Then $\mathbb{S}L$ is a differential $\mathbb{Z}$-graded Poisson algebra, where $\mathbb{S}L$ is the graded symmetric algebra of $L$.
\end{prop}
\begin{proof}
Denote $B:=\mathbf{U}(L)$, the graded universal enveloping algebra of $L$. Then $B$ has a canonical filtration such that the associated graded algebra is $\mathbb{S}L$ and $(\mathbb{S}L,d)$ is a differential graded algebra. Moreover, we can define the Poisson bracket by
$$\{a,b\}:=[a,b]$$ for any $a,b\in L$. Then $(\mathbb{S}L,d,\{,\})$ is a differential $\mathbb{Z}$-graded Poisson algebra.
\end{proof}

\begin{exa}\label{op}
{\rm Let $(A,\cdot, \{,\},d)$ be any differential graded Poisson algebra. Then $(A^{op},\cdot_{op}, \{,\}_{op},d_{op})$ is also a differential graded Poisson algebra, where
\begin{align*}
a\cdot_{op}b:=(-1)^{|a||b|}b\cdot a=a\cdot b,\\
d_{op}:=d,\\
\{a,b\}_{op}:=(-1)^{|a||b|}\{b,a\}=-\{a,b\},
\end{align*}
for any homogeneous elements $a,b\in A$.}
\end{exa}

\begin{prop}\label{tensor}
Let $(A,\cdot, \{,\}_A, d_A)$ and $(B,\ast,\{,\}_B,d_B)$ be any two differential graded Poisson algebras. Then $(A\otimes B,\star, \{,\},d)$ is a differential graded Poisson algebra, where
\begin{align*}
(a\otimes b)\star (a'\otimes b'):=(-1)^{|a'||b|}(a\cdot a')\otimes (b\ast b'),\\
d(a\otimes b):=d_A(a)\otimes b+(-1)^{|a|}a\otimes d_B(b),\\
\{a\otimes b, a'\otimes b'\}:=(-1)^{|a'||b|}(\{a,a'\}_A\otimes (b\ast b')+(a\cdot a')\otimes\{b,b'\}_B),\end{align*}
for any homogeneous elements $a,a'\in A$ and $b,b'\in B$. Moreover, the category of differential graded Poisson algebras is a symmetric monoidal category with a left and right identity $\Bbbk$.
\end{prop}

\begin{proof}
Let $a\otimes b, a'\otimes b', a''\otimes b''$ be any homogeneous elements of $A\otimes B$. We will check each axiom of Definition \ref{dgpa}. For the sake of simplicity, we omit $\cdot$, $\ast$ and the subscripts here.

The graded antisymmetry is immediate from
\begin{eqnarray*}
\{a\otimes b, a'\otimes b'\}&=&(-1)^{|a'||b|}((aa')\otimes\{b,b'\}+\{a,a'\}\otimes (bb'))\\
&=&-(-1)^{|a'||b|+|a||a'|+|b||b'|}((a'a)\otimes\{b',b\}+\{a',a\}\otimes (b'b))\\
&=&-(-1)^{|a\otimes b||a'\otimes b|}((-1)^{|a||b'|}(a'a)\otimes\{b',b\}+\{a',a\}\otimes (b'b))\\
&=&-(-1)^{|a\otimes b||a'\otimes b|}\{a'\otimes b',a\otimes b\}.
\end{eqnarray*}
Note that
\begin{eqnarray*}
\{a\otimes b,\{a'\otimes b',a''\otimes b''\}\}
&=&(-1)^{|a''||b'|}\{a\otimes b, (a'a'')\otimes \{b',b''\}+\{a',a''\}\otimes (b'b'')\}\\
&=& (-1)^{|a''||b'|}(\{a\otimes b,(a'a'')\otimes \{b',b''\}\}+\{a\otimes b,\{a',a''\}\otimes (b'b''))\\
&=&(-1)^{|a''||b'|+|b|(|a'|+|a''|)}((aa'a'')\otimes\{b,\{b',b''\}\}+\{a,a'a''\}\otimes (b\{b',b''\}\\
&&+(a\{a',a''\})\otimes\{b,b'b''\}+\{a,\{a',a''\}\}\otimes(bb'b'')),
\end{eqnarray*}
\begin{eqnarray*}
\{\{a\otimes b, a'\otimes b'\},a''\otimes b''\}
&=&(-1)^{|a'||b|}\{(aa')\otimes \{b,b'\}+\{a,a'\}\otimes (bb'), a''\otimes b''\}\\
&=&(-1)^{|a'||b|}(\{(aa')\otimes \{b,b'\}, a''\otimes b''\}+\{\{a,a'\}\otimes (bb'), a''\otimes b''\})\\
&=&(-1)^{|a'||b|+|a''|(|b|+|b'|)}((aa'a'')\otimes \{\{b,b'\},b''\}+\{aa', a''\}\otimes (\{b,b'\}b'')\\
&&+(\{a,a'\}a'')\otimes \{bb',b''\}+\{\{a,a'\},a''\}\otimes bb'b''),
\end{eqnarray*}
\begin{eqnarray*}
&&(-1)^{(|a|+|b|)(|a'|+|b'|)}\{a'\otimes b',\{a\otimes b,a''\otimes b''\}\}\\
&=&(-1)^{(|a|+|b|)(|a'|+|b'|)+|b||a''|}\{a'\otimes b',(aa'')\otimes\{b,b''\}+\{a,a''\}\otimes (bb'')\}\\
&=&(-1)^{(|a|+|b|)(|a'|+|b'|)+|b||a''|+|b'|(|a|+|a''|)}((a'aa'')\otimes \{b',\{b,b''\}\}\\
&&+\{a',aa''\}\otimes (b'\{b,b''\})+(a'\{a,a''\})\otimes\{b',bb''\}+\{a',\{a,a''\}\}\otimes (b'bb'')),
\end{eqnarray*}
\begin{eqnarray*}
&&(-1)^{(|a|+|b|)(|a'|+|b'|)+|b||a''|+|b'|(|a|+|a''|)}((a'aa'')\otimes \{b',\{b,b''\}\}\\
&=&(-1)^{|a'||b|+|a''||b|+|a''||b'|}((aa'a'')\otimes(\{b,\{b',b''\}\}-\{\{b,b'\},b''\}))
\end{eqnarray*}
and
\begin{eqnarray*}
&&(-1)^{(|a|+|b|)(|a'|+|b'|)+|b||a''|+|b'|(|a|+|a''|)}(\{a',\{a,a''\}\}\otimes (b'bb'')\\
&=&(-1)^{|a'||b|+|a''||b|+|a''||b'|}((\{a,\{a',a''\}\}-\{\{a,a'\},a''\})\otimes(bb'b'')),
\end{eqnarray*}
thus in order to prove
$$\{a\otimes b,\{a'\otimes b',a''\otimes b''\}\}=\{\{a\otimes b,a'\otimes b'\},a''\otimes b''\}+(-1)^{(|a|+|b|)(|a'|+|b'|)}\{a'\otimes b',\{a\otimes b,a''\otimes b''\}\},$$it suffices to prove
\begin{eqnarray*}
&&(-1)^{|a'||b|+|a''||b|+|a''||b'|}(\{a,a'a''\}\otimes b\{b',b''\}+a\{a',a''\}\otimes \{b,b'b''\})\\
&=&(-1)^{|a'||b|+|a''||b|+|a''||b'|}[(\{aa',a''\}\otimes\{b,b'\}b''+\{a,a'\}a''\otimes \{bb',b''\})+(-1)^{|a||a'|+|b||b'|}\{a',aa''\}\otimes b'\{b,b''\}\\
&&+(-1)^{|a||a'|+|b||b'|}a'\{a,a''\}\otimes \{b',bb''\}].
\end{eqnarray*}
Note that
\begin{eqnarray*}
&&(-1)^{|a'||b|+|a''||b|+|a''||b'|}[(\{aa',a''\}\otimes\{b,b'\}b''+\{a,a'\}a''\otimes \{bb',b''\})+(-1)^{|a||a'|+|b||b'|}\{a',aa''\}\otimes b'\{b,b''\}\\
&&+(-1)^{|a||a'|+|b||b'|}a'\{a,a''\}\otimes \{b',bb''\}]\\
&=&(-1)^{|a'||b|+|a''||b|+|a''||b'|}[(a\{a',a''\}+(-1)^{|a||a'|}a'\{a,a''\})\otimes \{b,b'\}b''+\{a,a'\}a''\otimes (b\{b',b''\}+(-1)^{|b||b'|}b'\{b,b''\} )\\
&&+(-1)^{|a||a'|+|b||b'|}((\{a',a\}a''+(-1)^{|a||a'|}a\{a',a''\})\otimes b'\{b.b''\}+a'\{a,a''\}\otimes (\{b'b\}b''+(-1)^{|b||b'|}b\{b',b''\}))]\\
&=&(-1)^{|a'||b|+|a''||b|+|a''||b'|}[a\{a',a''\}\otimes \{b,b'\}b''+(-1)^{|a||a'|}a'\{a,a''\}\otimes \{b,b'\}b''+\{a,a'\}a''\otimes b\{b',b''\}\\
&&+(-1)^{|b||b'|}\{a,a'\}a''\otimes b'\{b,b''\}+(-1)^{|a||a'|+|b||b'|}\{a',a\}a''\otimes b'\{b,b''\}+(-1)^{|b||b'|}a\{a',a''\}\otimes b'\{b,b''\}\\
&&+(-1)^{|a||a'|+|b||b'|}a'\{a,a''\}\otimes \{b',b\}b''+(-1)^{|a||a'|}a'\{a,a''\}\otimes b\{b',b''\}]\\
&=&(-1)^{|a'||b|+|a''||b|+|a''||b'|}[(a\{a',a''\}\otimes \{b,b'\}b''+(-1)^{|b||b'|}a\{a',a''\}\otimes b'\{b,b''\})\\
&&+(\{a,a'\}a''\otimes b\{b',b''\}+(-1)^{|a||a'|}a'\{a,a''\}\otimes b\{b',b''\})]\\
&=&(-1)^{|a'||b|+|a''||b|+|a''||b'|}(\{a,a'a''\}\otimes b\{b',b''\}+a\{a',a''\}\otimes \{b,b'b''\}).
\end{eqnarray*}
Therefore, $(A\otimes B, \{,\})$ is a graded Lie algebra.

$(A\otimes B, \star)$ is graded commutative since
\begin{eqnarray*}
(a\otimes b)\star (a'\otimes b')&=&(-1)^{|a'||b|}aa'\otimes bb'\\
&=&(-1)^{|a'||b|+|a||a'|+|b||b'|}a'a\otimes b'b\\
&=&(-1)^{(|a|+|b|)(|a'|+|b'|)}(a'\otimes b')\star (a\otimes b).\end{eqnarray*}

Biderivation property follows from
\begin{eqnarray*}
\{a\otimes b,(a'\otimes b')\star (a''\otimes b'')\}&=&
\{a\otimes b,(-1)^{|a''||b'|}a'a''\otimes b'b''\}\\
&=&(-1)^{|a''||b'|+|b||a'|+|b||a''|}(aa'a''\otimes \{b,b'b''\}+\{a,a'a''\}\otimes bb'b''),
\end{eqnarray*}
\begin{eqnarray*}
\{a\otimes b,a'\otimes b'\}\star(a''\otimes b'')&=&(-1)^{|a'||b|}(aa'\otimes\{b,b'\}+\{a,a'\}\otimes bb')\star(a''\otimes b'')\\
&=&(-1)^{|a''||b'|+|b||a'|+|b||a''|}(aa'a''\otimes \{b,b'\}b''+\{a,a'\}a''\otimes bb'b'')
\end{eqnarray*}
and
\begin{eqnarray*}
(-1)^{(|a\otimes b||a'\otimes b'|}(a'\otimes b')\star\{a\otimes b,a''\otimes b''\}
&=&(-1)^{(|a|+|b|)(|a'|+|b'|)}(a'\otimes b')\star\{a\otimes b,a''\otimes b''\}\\&=&(-1)^{|a''||b'|+|b||a'|+|b||a''|}(aa'a''\otimes \{b,b'b''\}-aa'a''\otimes\{b,b'\}b''\\
&&-\{a,a'\}a''\otimes bb'b''+\{a,a'a''\}\otimes bb'b'').
\end{eqnarray*}

$d^2=0$ follows from
\begin{eqnarray*}
d(d(a\otimes b))&=&d(da\otimes b+(-1)^{|a|}a\otimes db)\\
&=&d^2a\otimes b+(-1)^{|a|+1}da\times db+(-1)^{|a|}da\otimes db_(-1)^{|a|+|a|}a\otimes d^2b\\
&=&0.
\end{eqnarray*}
The graded Leibniz rule for product is true since
\begin{eqnarray*}
d((a\otimes b)\star(a'\otimes b'))&=&(-1)^{|a'||b|}d(aa'\otimes bb')\\
&=&(-1)^{|a'||b|}d(aa')\otimes bb'+(-1)^{|a'||b|+|a|+|a'|}aa'\otimes d(bb')\\
&=&(-1)^{|a'||b|}(da)a'\otimes bb'+(-1)^{|a'||b|+|a|}a(da')\otimes bb'\\
&&+(-1)^{|a'||b|+|a|+|a'|}aa'\otimes(db)b'+(-1)^{|a'||b|+|a|+|a'|+|b|}aa'\otimes b(db')\\
&=&((-1)^{|a'||b|}(da)a'\otimes bb'+(-1)^{|a'||b|+|a|+|a'|}aa'\otimes(db)b')\\
&&+((-1)^{|a'||b|+|a|}a(da')\otimes bb'+(-1)^{|a'||b|+|a|+|a'|+|b|}aa'\otimes b(db'))\\
&=&d(a\otimes b)\star (a'\otimes b')+(-1)^{|a|+|b|}(a\otimes b)\star d(a'\otimes b').
\end{eqnarray*}
The graded Leibniz rule for bracket is true since
\begin{eqnarray*}
d\{a\otimes b,a'\otimes b'\}&=&(-1)^{|a'||b|}d(aa'\otimes \{b,b'\}+\{a,a'\}\otimes bb')\\
&=&(-1)^{|a'||b|}d(aa')\otimes \{b,b'\}+(-1)^{|a'||b|+|a|+|a'|}aa'\otimes d(\{b,b'\})\\
&&+(-1)^{|a'||b|}d(\{a,a'\})\otimes bb'+(-1)^{|a'||b|+|a|+|a'|}\{a,a'\}\otimes d(\{b,b'\})\\
&=&((-1)^{|a'||b|}(da)a'\otimes\{b,b'\}+(-1)^{|a'||b|+|a|+|a'|}aa'\otimes \{db,b'\}\\
&&+(-1)^{|a'||b|}\{da,a'\}\otimes bb'+(-1)^{|a'||b|+|a|+|a'|}\{a,a'\}\otimes (db)b')\\
&&+((-1)^{|a'||b|+|a|}a(da')\otimes\{b,b'\}+(-1)^{|a'||b|+|a|+|a'|+|b|}aa'\otimes\{b,db'\}\\
&&+(-1)^{|a'||b|+|a|}\{a,(da')\}\otimes bb'+(-1)^{|a'||b|+|a|+|a'|+|b|}\{a,a'\}\otimes b(db'))\\
&=&\{d(a\otimes b),a'\otimes b'\}+(-1)^{|a|+|b|}\{a\otimes b,d(a'\otimes b')\}.
\end{eqnarray*}
Therefore, $(A\otimes B,\star, \{,\},d)$ is a differential graded Poisson algebra.

In order to prove that the category of differential graded Poisson algebras is a symmetric monoidal category with a left and right identity $\Bbbk$, it suffices to prove the following map
\[\varphi: A\otimes B\to B\otimes A\]given by
\[\varphi(a\otimes b)=(-1)^{|a||b|}b\otimes a\] is a differential graded Poisson algebra isomorphism, where $A$ and $B$ are any two differential graded Poisson algebras and $a\in A, b\in B$ are any homogeneous elements. Note that $\varphi \circ\varphi=1$, we only need to check that $\varphi$ is a differential graded Poisson algebra map, which follows from
\begin{eqnarray*}
\varphi((a\otimes b)(a'\otimes b'))&=&(-1)^{|a'||b|}\varphi(aa'\otimes bb')\\
&=&(-1)^{|a'||b|+(|a|+|a'|)(|b|+|b'|)}bb'\otimes aa'\\
&=&(-1)^{|a||b|+|a'||b'|+|a||b'|}bb'\otimes aa'\\
&=&(-1)^{|a||b|}(b\otimes a)(-1)^{|a'||b'|}(b'\otimes a')\\
&=&\varphi(a\otimes b)\varphi(a'\otimes b'),
\end{eqnarray*}
\begin{eqnarray*}
d\varphi(a\otimes b)&=&(-1)^{|a||b|}d(b\otimes a)\\
&=&(-1)^{|a||b|}d(b)\otimes a+(-1)^{|a||b|+|b|}b\otimes d(a)\\
&=&(-1)^{|a|+|a|(|b|+1)}d(b)\otimes a+(-1)^{(|a|+1)|b|}b\otimes d(a)\\
&=&\varphi(d(a)\otimes b)+(-1)^{|a|}\varphi(a\otimes d(b))\\
&=&\varphi(d(a)\otimes b+(-1)^{|a|}a\otimes d(b))\\
&=&\varphi d(a\otimes b)
\end{eqnarray*}
and
\begin{eqnarray*}
\varphi(\{a\otimes b,a'\otimes b'\})&=&(-1)^{|a'||b|}\varphi(aa'\otimes\{b,b'\}+\{a,a'\}\otimes bb')\\
&=&(-1)^{|a'||b|+(|a|+|a'|)(|b|+|b'|)}(\{b,b'\}\otimes aa'+bb'\otimes \{a,a'\})\\
&=&(-1)^{|a||b|+|a'||b'|+|a||b'|}(\{b,b'\}\otimes aa'+bb'\otimes \{a,a'\})\\
&=&\{(-1)^{|a||b|}b\otimes a, (-1)^{|a'||b'|}b'\otimes a'\}\\
&=&\{\varphi(a\otimes b),\varphi(a'\otimes b')\}.
\end{eqnarray*}
\end{proof}

From the constructions of opposite algebra and tensor product of differential graded Poisson algebras, we have the following observation.

\begin{prop}\label{tensoropp}
{\it Let $A,B$ be any two differential graded Poisson algebras. Then $(A\otimes B)^{op}=A^{op}\otimes B^{op}.$ $\hfill\Box$}\end{prop}
\begin{defn}\label{sub}
Let $A$ be a differential $\mathbb{Z}$-graded Poisson algebra and $B$ be a $\mathbb{Z}$-graded subspace of $A$ that is closed under multiplication.
\begin{itemize}
  \item[(i)] If $B$ is a graded subalgebra of $A$ satisfying $d(B)\subseteq B$ and $\{B,B\}\subseteq B$, then $B$ is called a {\it differential $\mathbb{Z}$-graded Poisson subalgebra} of $A$.
  \item[(ii)]If $d(B), B\cdot A, A\cdot B$ and $\{A,B\}$ are all contained in $B$, then $B$ is called a {\it differential $\mathbb{Z}$-graded Poisson ideal} of $A$. \end{itemize}
\end{defn}

\begin{prop}\label{2.7}
Let $A$ be a differential $\mathbb{Z}$-graded Poisson algebra and $B\subseteq A$ be a differential $\mathbb{Z}$-graded Poisson ideal of $A$. Then the quotient algebra $A/B$ has a natural differential $\mathbb{Z}$-graded Poisson algebra structure:
$$d_{A/B}(a+B):=d_A(a)+B\;{\rm and}\;\{a+B,a'+B\}_{A/B}:=\{a,a'\}_A+B.$$
$\hfill\Box$\end{prop}

\begin{defn}\label{poissonmap}
Let $A$ and $B$ be two differential $\mathbb{Z}$-graded Poisson algebras. A graded algebra map $f:A\to B$ of degree $0$ is called a {\it differential $\mathbb{Z}$-graded Poisson algebra map} if $f\circ d_A=d_B\circ f$ and $f(\{a,b\}_A)=\{f(a),f(b)\}_B$ for all homogeneous elements $a,b\in A$. Moreover, we call $A$ and $B$ are {\it isomorphic} as differential $\mathbb{Z}$-graded Poisson algebras provided that there is a bijective differential $\mathbb{Z}$-graded Poisson algebra map from $A$ to $B$, denoted as $A\cong B$.

We denote by $\mathbf{DGPA}$ the category of differential graded Poisson algebras whose morphism space consists of differential graded Poisson algebra maps.
\end{defn}
\begin{prop}
Let $A, B\in\mathbf{DGPA}$ and $f: A\to B$ be a differential $\mathbb{Z}$-graded Poisson algebra map. Then the following are true:
\begin{enumerate}
\item $\ker f$ is a differential $\mathbb{Z}$-graded Poisson ideal of $A$.
\item $\im f$ is a differential $\mathbb{Z}$-graded Poisson subalgebra of $B$.
\item $A/\ker f \cong \im f$ as differential $\mathbb{Z}$-graded Poisson algebras.
\item If $A'\subseteq A$ is a differential $\mathbb{Z}$-graded Poisson subalgebra, then $f(A')$ is a differential $\mathbb{Z}$-graded Poisson subalgebra of $B$.
  \item If $I\subseteq A$ is a differential $\mathbb{Z}$-graded Poisson ideal, then $f(I)$ is a differential $\mathbb{Z}$-graded Poisson ideal of $f(A)$.
$\hfill\Box$
\end{enumerate}
\end{prop}
\begin{prop}\label{2.12}
Let $A\in\mathbf{DGPA}$. Then the cohomology algebra $\mathcal{H}(A)$  is a $\mathbb{Z}$-graded Poisson algebra. Moreover, if $B\in\mathbf{DGPA}$ such that $A\cong B$ as differential $\mathbb{Z}$-graded Poisson algebras, then $\mathcal{H}(A)\cong \mathcal{H}(B)$ as $\mathbb{Z}$-graded Poisson algebras.
$\hfill\Box$\end{prop}

\subsection{Differential graded Poisson modules}
In this subsection, we define the {\it differential $\mathbb{Z}$-graded Poisson modules} over a differential $\mathbb{Z}$-graded Poisson algebra and study some basic properties of such modules.

\begin{defn}
Let $(A,d)\in\mathbf{DGA}$ and $M$ a left $\mathbb{Z}$-graded module over $A$. We call $M$ is a left {\it differential graded module} over $A$ provided that there is a $\Bbbk$-linear map $\partial: M\to M$ of degree $1$ such that $\partial^2=0$ and $$\partial(am)=d(a)m+(-1)^{|a|}a\partial(m)$$ for all homogeneous elements $a\in A$ and $m\in M$. Here $\partial$ is also called the {\it differential} of $M$.

Let $M, N$ be two differential graded modules over $A$ and $f:M\to N$ be a graded $A$-module map of degree zero. Then $f$ is called a {\it differential graded module map} if $f$ commutes with the differentials.

We usually denote a differential graded module by $(M,\partial)$, or simply by $M$ if no confusions arise, and denote by $\mathbf{DG}(A)$ the category of differential graded modules over the differential graded algebra $A$, and whose morphism space consists of differential graded module maps.
\end{defn}

\begin{defn}\label{def2}
Let $(A,\cdot, \{,\}_A,d)\in\mathbf{DGPA}$ and $(M,\partial)\in\mathbf{DG}(A)$. We call $M$ a {\it left differential $\mathbb{Z}$-graded Poisson module} over $A$ provided that
\begin{itemize}
  \item[(i)]  $M$ is a {\it left $\mathbb{Z}$-graded Poisson module} over the $\mathbb{Z}$-graded Poisson algebra $A$. That is to say, there is a linear map $\{,\}_M: A\otimes M\to M$ of degree $0$ such that
  \begin{itemize}
  \item[(ia)] $\{a,bm\}_M=\{a,b\}_Am+(-1)^{|a||b|}b\{a,m\}_M$;
  \item[(ib)] $\{ab,m\}_M=a\{b,m\}_M+(-1)^{|a||b|}b\{a,m\}_M$, and
  \item[(ic)]$\{a,\{b,m\}_M\}_M=\{\{a,b\}_A,m\}_M+(-1)^{|a||b|}\{b,\{a,m\}_M\}_M$,
\end{itemize}
for all homogeneous elements $a,b\in A$ and $m\in M$.
  \item[(ii)] the linear map $\partial$ is compatible with the bracket $\{,\}_M$. That is, we have $$\partial(\{a,m\}_M)=\{d(a),m\}_M+(-1)^{|a|}\{a,\partial(m)\}_M,$$ for all homogeneous elements $a\in A$ and $m\in M$.
\end{itemize}
Similarly, a left differential graded module $M$ over a differential graded Poisson algebra $A$ is usually denoted by $(M,\{,\}_M,\partial)$, or simply by $M$ if there are no confusions.
\end{defn}

\begin{remark}
The notion of {\it right differential $\mathbb{Z}$-graded Poisson module} can be defined similarly. Moreover, we call $M$ a {\it differential $\mathbb{Z}$-graded Poisson bimodule} if and only if $M$ is both a left and a right differential $\mathbb{Z}$-graded Poisson module and$$(am)b=a(mb)\;\;{\rm and}\;\;\{a,\{m,b\}_M\}_M=\{\{a,m\}_M,b\}_M$$ for all homogeneous elements $a,b\in A,\;m\in M$.
\end{remark}

\begin{exa}\label{poissonmodule}
{\rm The following are some trivial examples of differential $\mathbb{Z}$-graded Poisson modules.
\begin{enumerate}
  \item Any left differential $\mathbb{Z}$-graded modules over differential $\mathbb{Z}$-graded commutative algebras can be considered as left differential $\mathbb{Z}$-graded Poisson modules with trivial Poisson structure.
\item Any left differential $\mathbb{Z}$-graded Lie modules over differential $\mathbb{Z}$-graded Lie algebras can be considered as left differential $\mathbb{Z}$-graded Poisson modules with trivial product.
  \item Any left $\mathbb{Z}$-graded Poisson modules over $\mathbb{Z}$-graded Poisson algebras can be considered as left differential $\mathbb{Z}$-graded Poisson modules with trivial differentials.

  \item Any differential $\mathbb{Z}$-graded Poisson algebras can be regarded as left (right, bi) differential $\mathbb{Z}$-graded Poisson modules over themselves.
  \end{enumerate}}
\end{exa}

\begin{exa}\label{2.17}
{\rm Let $A\in\mathbf{DGPA}$, and  $(M,\ast,\{,\}_M,\partial)$ be a left differential $\mathbb{Z}$-graded Poisson module over $A$. Then $(M^{op},\ast_{op},\{,\}^{op}_M,\partial_{op})$ can be regarded naturally as a right differential $\mathbb{Z}$-graded Poisson $A^{op}$-module, where
$$m\ast_{op}a:=(-1)^{|a||m|}a\ast m,\;\{m,a\}^{op}_M:=(-1)^{|a||m|}\{a,m\}_M\;{\rm and}\; \partial_{op}:=\partial.$$
}\end{exa}

\begin{exa}\label{tensormodule}
{\rm Let $A,B\in\mathbf{DGPA}$. For any left differential $\mathbb{Z}$-graded Poisson $A$-module $M$ and any differential $\mathbb{Z}$-graded Poisson $B$-module $N$. Then similar to the proof of Proposition \ref{tensor}, we can get that $M\otimes N$  is a left differential $\mathbb{Z}$-graded Poisson $A\otimes B$-module via\begin{align*}
(a\otimes b)\ast (m\otimes n):=(-1)^{|b||m|}am\otimes bn,\\
\partial(m\otimes n):=\partial_M(m)\otimes n+(-1)^{|m|}m\otimes \partial_N(n),\\
\{a\otimes b,m\otimes n\}:=(-1)^{|b||m|}(am\otimes\{b,n\}_N+\{a,m\}_M\otimes bn),
\end{align*}
where $a\in A$, $b\in B$, $m\in M$ and $n\in N$ are any homogeneous elements.
}\end{exa}
\begin{prop}
Let $M$ and $N$ be two left differential $\mathbb{Z}$-graded Poisson modules over differential $\mathbb{Z}$-graded Poisson algebras $A$ and $B$, respectively. Then $(M\otimes N)^{op}\cong M^{op}\otimes N^{op}$ as right differential $\mathbb{Z}$-graded Poisson $(A\otimes B)^{op}$-modules.
$\hfill\Box$\end{prop}

\begin{defn}
Let $A\in\mathbf{DGPA}$, and $M$ be a left differential $\mathbb{Z}$-graded Poisson module over $A$. A left  $\mathbb{Z}$-graded submodule $N\leqslant M$ is called a {\it  left differential $\mathbb{Z}$-graded Poisson submodule}  provided that  $\partial (N)\subseteq N$ and $\{A,N\}_M\subseteq N$, which is usually denoted by $N\leqslant_{P}M$.
\end{defn}
\begin{prop}\label{2.16}
Let $M$ be a left differential $\mathbb{Z}$-graded Poisson module and $N\leqslant_{P}M$. Then the left quotient $\mathbb{Z}$-graded $A$-module $M/N$ also has a natural left differential $\mathbb{Z}$-graded Poisson module structure:
$$\partial_{M/N}(m+N):=\partial(m)+N\;{\rm and}\; \{a,m+N\}_{M/N}:=\{a,m\}+N$$ for all homogeneous elements $a\in A$ and $m\in M$.
$\hfill\Box$\end{prop}

\begin{defn}
Let $A\in\mathbf{DGPA}$ and $M,N$ be two left differential $\mathbb{Z}$-graded Poisson $A$-modules. A left $\mathbb{Z}$-graded $A$-module map $f: M\to N$ of degree $0$ is called a {\it left differential $\mathbb{Z}$-graded Poisson module map} provided that
$$f(\partial_M(m))=\partial_N(f(m))\;{\rm and}\;f(\{a,m\}_M)=\{a,f(m)\}_N$$for all homogeneous elements $a\in A$ and $m\in M$. In particular, we say $M$ and $N$ are {\it isomorphic} as left differential $\mathbb{Z}$-graded Poisson modules if there is a bijective left differential $\mathbb{Z}$-graded Poisson module map $f:M\to N$, which is usually denoted by $M\cong N$.

Let $\mathbf{ DGP}(A)$ denote the category of left differential $\mathbb{Z}$-graded Poisson $A$-modules whose morphism space consists of left differential graded Poisson module maps.
\end{defn}

\begin{prop}\label{2.20}
Let $A\in\mathbf{DGPA}$, $M,N\in\mathbf{ DGP}(A)$ and $f:M\to N$ be a left differential $\mathbb{Z}$-graded Poisson module map. Then
\begin{enumerate}
\item $\ker f\leqslant_{P}M$;
\item $f$ preserves differential $\mathbb{Z}$-graded Poisson submodules. More precisely, if $X\leqslant_{P}M$, then $f(X)\leqslant_{P}N$. In particular, $\im f\leqslant_{P}N$;
\item $M/\ker f\cong \im f$ as differential $\mathbb{Z}$-graded Poisson modules.$\hfill\Box$\end{enumerate}
\end{prop}

\begin{prop}\label{2.22}
Let $A\in\mathbf{DGPA}$ and $M\in\mathbf{DGP}(A)$. Then the cohomology $\mathrm{H}(M)$ of $M$ is a left $\mathbb{Z}$-graded Poisson module over $\mathcal{H}(A)$. Moreover, if $N\in\mathbf{DGP}(A)$ such that $M\cong N$ as differential $\mathbb{Z}$-graded Poisson modules, then $\mathrm{H}(M)\cong \mathrm{H}(N)$ as $\mathbb{Z}$-graded Poisson $\mathcal{H}(A)$-modules.$\hfill\Box$\end{prop}

\medskip
\section{Universal enveloping algebras of differential graded Poisson algebras}
In this section, we study the universal enveloping algebras of differential $\mathbb{Z}$-graded Poisson algebras. \subsection{Constructions of $A^e$ and $A^E$}
Let $A\in\mathbf{DGPA}$. In this subsection, we construct two isomorphic differential graded algebras from $A$.
\subsubsection{Construction of $A^e$}
Let $(A,d,\{,\})\in\mathbf{DGPA}$, $m_A:=\{m_a|\; a\in A\}$ and $h_A:=\{h_a|\; a\in A\}$ be two copies of the $\mathbb{Z}$-graded vector space $A$ endowed with the following two $\Bbbk$-linear isomorphisms:
$$m: A\to m_A\;\;\;{\rm and}\;\;\;h:A\to h_A,$$sending any element $a\in A$ to $m_a$ and $h_a$, respectively. For any homogeneous element $a\in A$, we suppose that $m_a$ and $h_a$ are both homogeneous elements with $|h_a|=|m_a|=|a|$. We define a $\mathbb{Z}$-graded algebra $A^e$ generated by $m_A$ and $h_A$, subject to the following relations:
\begin{align*}
m_{ab}=m_am_b,\\
h_{ab}=m_ah_b+(-1)^{|a||b|}m_bh_a,\\
m_{\{a,b\}}=h_am_b-(-1)^{|a||b|}m_bh_a,\\
h_{\{a,b\}}=h_ah_b-(-1)^{|a||b|}h_bh_a,\\
m_1=1,
\end{align*}
for any homogeneous elements $a,b\in A$.

\begin{lemma}\label{lem1}
$A^e$ is a differential $\mathbb{Z}$-graded algebra.
\end{lemma}
\begin{proof}
By the construction of $A^e$, it is easy to see that $A^e$ is a $\mathbb{Z}$-graded algebra. Let $\partial: A^e\to A^e$ be a $\Bbbk$-linear map of degree 1 determined by
$$\partial(m_a):=m_{d(a)},\;\;\;\partial(h_a):=h_{d(a)},$$
for any homogeneous element $a\in A$ and the graded Leibniz rule, i.e., for any homogeneous elements $x,y\in A^e$, we have $$\partial(xy)=\partial(x)y+(-1)^{|x||y|}x\partial(y).$$
Note that $\partial^2=0$ is immediate from the definition. Thus it suffices to prove the following equations:
\begin{align*}
\partial(m_{ab}-m_am_b)=0,\\
\partial(h_{ab}-m_ah_b-(-1)^{|a||b|}m_bh_a)=0,\\
\partial(m_{\{a,b\}}-h_am_b+(-1)^{|a||b|}m_bh_a)=0,\\
\partial(h_{\{a,b\}}-h_ah_b+(-1)^{|a||b|}h_bh_a)=0,
\end{align*}
which follow from
\begin{eqnarray*}
\partial(m_{ab})&=&m_{d(ab)}\\
&=&m_{d(a)b+(-1)^{|a|}ad(b)}\\
&=&m_{d(a)}m_b+(-1)^{|a|}m_am_{d(b)}\\
&=&\partial(m_a)m_b+(-1)^{|a|}m_a\partial(m_b)\\
&=&\partial(m_am_b),
\end{eqnarray*}
\begin{eqnarray*}
\partial(h_{ab})&=&h_{d(ab)}\\
&=&h_{d(a)b+(-1)^{|a|}ad(b)}\\
&=&h_{d(a)b}+(-1)^{|a|}h_{ad(b)}\\
&=&m_{d(a)}h_b+(-1)^{|d(a)||b|}m_bh_{d(a)}+(-1)^{|a|}(m_ah_{d(b)}+(-1)^{|a||d(b)|}m_{d(b)}h_a)\\
&=&(m_{d(a)}h_b+(-1)^{|a|}m_ah_{d(b)})+(-1)^{|a||b|}(m_{d(b)}h_a+(-1)^{|b|}m_bh_{d(a)})\\
&=&\partial(m_ah_b)+(-1)^{|a||b|}\partial(m_bh_a)\\
&=&\partial(m_ah_b+(-1)^{|a||b|}m_bh_a),
\end{eqnarray*}
\begin{eqnarray*}
\partial(m_{\{a,b\}})&=&m_{d(\{a,b\})}\\
&=&m_{\{d(a),b\}+(-1)^{|a|}\{a,d(b)\}}\\
&=&m_{\{d(a),b\}}+(-1)^{|a|}m_{\{a,d(b)\}}\\
&=&h_{d(a)}m_b-(1)^{|d(a)||b|}m_bh_{d(a)}+(-1)^{|a|}(h_am_{d(b)}-(-1)^{|a||d(b)|}m_{d(b)}h_a)\\
&=&(h_{d(a)}m_b+(-1)^{|a|}h_am_{d(b)})-(-1)^{|a||b|}(m_{d(b)}h_a-(-1)^{|b|}m_bh_{d(a)})\\
&=&\partial(h_am_b)-(-1)^{|a||b|}\partial(m_bh_a)\\
&=&\partial(h_am_b-(-1)^{|a||b|}m_bh_a)
\end{eqnarray*}
and
\begin{eqnarray*}
\partial(h_{\{a,b\}})&=&h_{d(\{a,b\})}\\
&=&h_{\{d(a),b\}+(-1)^{|a|}\{a,d(b)\}}\\
&=&h_{d(a)}h_b-(-1)^{|d(a)||b|}h_bh_{d(a)}+(-1)^{|a|}(h_ah_{d(b)}-(-1)^{|a||d(b)|}h_{d(b)}h_a)\\
&=& (h_{d(a)}h_b+(-1)^{|a|}h_ah_{d(b)})-(-1)^{|a||b|}(h_{d(b)}h_a+(-1)^{|b|}h_bh_{d(a)})\\
&=& \partial(h_ah_b)-(-1)^{|a||b|}\partial(h_bh_a)\\
&=& \partial(h_ah_b-(-1)^{|a||b|}h_bh_a).
\end{eqnarray*}
\end{proof}

\subsubsection{Construction of $A^E$}
We first recall the smash product in the ``differential graded'' setting. Let $(H,m,u,\\\Delta,\epsilon,S,\partial_H)$ be a differential graded Hopf algebra with the standard notations. By Sweedler notation, we write
\[
\Delta(a)=\sum\limits_{(a)} a_1\otimes a_2,
\]
for all homogenous element $a\in A$. Note that $\Delta$ admits the Koszul sign, i.e.,
\[\Delta(ab)=\sum_{(a),(b)} (-1)^{|a_2||b_1|}a_1b_1\otimes a_2b_2,\]
for all homogenous elements $a,b\in H$.

A differential graded algebra $B$ is said to be an $H$-module differential graded algebra if
\begin{itemize}
\item $B$ is a left differential graded $H$-module via $h\otimes a\mapsto h\cdot a$;
\item $h\cdot 1_B=\epsilon(h)1_B$; and
\item $h\cdot(ab)=\sum\limits_{(h)} (-1)^{|a||h_2|}(h_1\cdot a)(h_2\cdot b)$, for all homogenous elements $h\in H$ and $a,b\in B$.
\end{itemize}

Suppose that $(B,\partial_B)$ is a left $H$-module differential graded algebra, then the smash product $B\#H$ admits a natural differential graded algebra structure:
\begin{itemize}
\item $B\#H\cong B\otimes H$ as $\Bbbk$-vector spaces;
\item the product is given by
\[
(a\#g)(b\#h)=\sum\limits_{(g)} (-1)^{|g_2||b|}a(g_1\cdot b)\# (g_2h);\
\]and
\item the differential $d$ is given by $d(a\#h)=\partial_B(a)\#h+(-1)^{|a|}a\#\partial_H(h)$, for all homogenous elements $a,b\in B$ and $g,h\in H$.
\end{itemize}

Note that in the smash product $B\#H$, both $B$ and $H$ are differential graded subalgebras via the embeddings $B\#1$ and $1\#H$. Moreover, $B\#H$ is generated by $B$ and $H$ as differential graded algebras.

Now we return to the construction of $A^E$, where $A\in\mathbf{DGPA}$. We denote by $\mathbf{U}(A)$ the universal enveloping algebra of the differential graded Lie algebra $(A,\{,\})$. In order to avoid the confusions, we use $\mathbb{A}$ to denote the image of $A$ under the natural projection $T(A)\to \mathbf{U}(A)$. It is well-known that $\mathbf U(A)$ is a differential graded Hopf algebra, whose comultiplication is given by $\Delta(h)=h\otimes 1+1\otimes h$ for all $h\in \mathbb{A}$. We first construct a left $\mathbf U(A)$-module structure on $A$ via $1_{\mathbf U(A)}\cdot a:=a$ and $h\cdot a:=\{h,a\}$ for all $h\in \mathbb{A}$ and $a\in A$.

\begin{lemma}
Let $A\in\mathbf{DGPA}$. Then $A$ is a left $\mathbf U(A)$-module differential graded algebra.
\end{lemma}
\begin{proof}
Firstly, we show that the left action of $\mathbf U(A)$ on $A$ is well defined.
According to the definition, we have
\begin{align*}
(g\otimes h-(-1)^{|g||h|}h\otimes g)\cdot a=\{g,\{h,a\}\}-(-1)^{|g||h|}\{h,\{g,a\}\}=\{\{g,h\},a\}=\{g,h\}\cdot a,
\end{align*}
for all homogenous elements $g,h\in \mathbb{A}$ and $a\in A$ by the graded Jacobi identity in $A$.
Recall that
$$\mathbf U(A):=\frac{T(\mathbb{A})}{(g\otimes h-(-1)^{|g||h|}h\otimes g-\{g,h\})}$$ for all homogenous elements $g,h\in \mathbb{A}$. Hence, the action is well-defined.

Secondly, we show that the action respects the multiplication in $A$. Let $h$ be any homogenous element in $\mathbb{A}$. Since $\Delta(h)=h\otimes 1+1\otimes h$ , we have $\epsilon(h)=0$. Then $h\cdot 1_A=\{h,1_A\}=0=\epsilon(h)1_A$. Suppose we have $a\cdot 1_A=\epsilon(a)1_A,b\cdot 1_A=\epsilon(b)1_A$. Because $\epsilon$ is a differential graded algebra map, we have
$$(ab)\cdot 1_A=a\cdot (b\cdot 1_A)=a\cdot \epsilon(b)1_A=\epsilon(a)\epsilon(b)1_A=\epsilon(ab)1_A.$$ Hence $a\cdot 1_A=\epsilon(a)1_A$ for all $a\in \mathbf U(A)$ by induction since it is verified on the generating space $\mathbb{A}$.

On the other side, we see that
\begin{eqnarray*}
h\cdot(ab)&=&\{h,ab\}\\
&=&\{h,a\}b+(-1)^{|a||h|}a\{h,b\}\\
&=&(h\cdot a)(1\cdot b)+(-1)^{|a||h|}(1\cdot a)(h\cdot b)\\
&=&\sum_{(h)} (-1)^{|a||h_2|}(h_1\cdot a)(h_2\cdot b),
\end{eqnarray*}
for any homogenous elements $a,b\in A$ since $\Delta(h)=h\otimes 1+1\otimes h$. Now suppose that $$f\cdot(ab)=\sum_{(f)} (-1)^{|a||f_2|}(f_1\cdot a)(f_2\cdot b)$$ and $$g\cdot(ab)=\sum_{(g)} (-1)^{|a||g_2|}(g_1\cdot a)(g_2\cdot b)$$ for homogenous elements $f,g\in \mathbf U(A)$ and $a,b\in A$. Then we have
\begin{eqnarray*}
(fg)\cdot(ab)&=&f\cdot \left(\sum_{(g)}(-1)^{|g_2||a|}(g_1\cdot a)(g_2\cdot b)\right)\\
&=&\sum_{(f),(g)}(-1)^{|g_2||a|+|f_2|(|g_1|+|a|)}\left(f_1\cdot(g_1\cdot a)\right)\left(f_2\cdot (g_2\cdot b)\right)\\
&=&\sum_{(f),(g)}(-1)^{|a|(|f_2|+|g_2|)+|f_2||g_1|}\left((f_1g_1)\cdot a)\right)\left((f_2g_2)\cdot b)\right)\\
&=&\sum_{(fg)}(-1)^{|a||(fg)_2|}\left((fg)_1\cdot a)\right)\left((fg)_2\cdot b)\right).
\end{eqnarray*}
Hence, it follows by induction that $\varphi\cdot(ab)=\sum\limits_{(\varphi)} (-1)^{|a||\varphi_2|}(\varphi_1\cdot a)(\varphi_2\cdot b)$ for all homogenous elements $\varphi\in \mathbf U(A)$ and $a,b\in A$.

Finally, we want to show that the differential respects the action as a differential graded $\mathbf U(A)$-module. For any homogenous element $h\in \mathbb{A}$, we have $$\partial_A(h\cdot a)=\partial_A\{h,a\}=\{\partial_A(h),a\}+(-1)^{|h|}\{h,\partial_A(A)\}.$$ Note that $\partial_{\mathbf U(A)}(h)=\partial_A(h)$, we have $\partial_A(h\cdot a)=\partial_{\mathbf U(A)}(h)\cdot a+(-1)^{|h|} h\cdot \partial_A(a)$ for any $a\in A$. Now suppose that $$\partial_A(g\cdot a)=\partial_{\mathbf U(A)}(g)\cdot a+(-1)^{|g|} g\cdot \partial_A(a)$$ and $$\partial_A(h\cdot a)=\partial_{\mathbf U(A)}(h)\cdot a+(-1)^{|h|} h\cdot \partial_A(a)$$ for any homogenous elements $g,h\in \mathbf U(A)$ and $a\in A$. Then we have
\begin{eqnarray*}
\partial_A((hg)\cdot a)&=&\partial_A(h\cdot(g\cdot a))\\
&=&\partial_{\mathbf U(A)}(h)\cdot (g\cdot a)+(-1)^{|h|}h\cdot \partial_A(g\cdot a)\\
&=&(\partial_{\mathbf U(A)}(h)g)\cdot a+(-1)^{|h|}(h\partial_{\mathbf U(A)}(g))\cdot a+(-1)^{|h|+|g|}(hg)\cdot \partial_A(a)\\
&=&\partial_{\mathbf U(A)} (hg)\cdot a+(-1)^{|hg|}(hg)\cdot \partial_A(a).
\end{eqnarray*}
Hence by induction it follows that $\partial_A(h\cdot a)=\partial_{\mathbf U(A)}(h)\cdot a+(-1)^{|h|} h\cdot \partial_A(a)$ for all homogenous elements $h\in \mathbf U(A)$ and $a\in A$ since it is true on the generating space $\mathbb{A}$.
\end{proof}

Then, we define
\[A^E:=\frac{A\#\mathbf U(A)}{(1\#(ab)-a\#b-(-1)^{|a||b|}b\#a)}\]
for any homogenous elements $a,b\in A\;{\rm or}\;\mathbb{A}$.

\begin{remark}
Note that, for any homogeneous element $a$, when we use the phase `` $a\in A\;{\rm or}\;\mathbb{A}$'', it means that $a\in A$ if $a$ lies on the left side of ``$\#$'', and $a\in\mathbb{A}$ otherwise.
\end{remark}

\begin{prop}\label{dg}
$A^E$ is a $\mathbb Z$-graded differential graded algebra.
\end{prop}
\begin{proof}
It suffices to show that the ideal $I$ generated by $$1\#(ab)-a\#b-(-1)^{|a||b|}b\#a$$ for all homogenous elements $a,b\in A\;{\rm or}\;\mathbb{A}$ is a differential graded ideal in $A\#\mathbf U(A)$, which follows from
\begin{eqnarray*}
d(1\#(ab)-a\#b-(-1)^{|a||b|}b\#a)&=&1\#d_{\mathbf U(A)}(ab)-d(a\#b)-(-1)^{|a||b|}d(b\#a)\\
&=&1\#d_A(a)b+(-1)^{|a|}1\#ad_A(b)-d_A(a)\#b-(-1)^{|a|}a\#d_A(b)\\
&& -(-1)^{|a||b|}d_A(b)\#a-(-1)^{|a||b|+|b|}b\#d_A(a)\\
&=&\left(1\#d_A(a)b-d_A(a)\#b-(-1)^{|d_Aa||b|}b\#d_A(a)\right)\\
&& +(-1)^{|a|}\left(1\#ad_A(b)-a\#d_A(b)-(-)^{|a||d_Ab|}d_A(b)\#a\right)\\
&\in& I.
\end{eqnarray*}
\end{proof}

\subsubsection{The relationship between $A^e$ and $A^E$}
We first define a differential graded algebra map $\alpha: A^e\to A^E$ by
\begin{align*}
\alpha(m_a)=\overline{a\#1},\ \alpha(h_a)=\overline{1\#a},\ \mbox{for all}\ a\in A.
\end{align*}
\begin{lemma}
The map $\alpha:A^e\to A^E$ is a well-defined differential graded algebra map.
\end{lemma}
\begin{proof}
Since $A^e$ is generated by $m_A$ and $h_A$ as differential graded algebras. It suffices to show that $\alpha$ preserves the relations in $A^e$ and commutes with differentials. Firstly, we use the fact that $A$ and $\mathbf U(A)$ are differential graded subalgebras of the smash product $A\#\mathbf U(A)$. Thus we have
\begin{align*}
\alpha(m_1)&=\overline{1\#1}=1_{A^E}=\alpha(1_{A^e}),\\
\alpha(m_{ab})&=\overline{(ab)\#1}=\overline{(a\#1)(b\#1)}=\alpha(m_a)\alpha(m_b)=\alpha(m_am_b),\\
\alpha(h_{\{a,b\}})&=\overline{1\#\{a,b\}}=\overline{1\#(a\otimes b-(-1)^{|a||b|}b\otimes a)}=\overline{(1\#a)(1\#b)}-\overline{(-1)^{|a||b|}(1\#b)(1\#a)}=\alpha(h_ah_b-(-)^{|a||b|}h_bh_b),
\end{align*}
for all homogenous elements $a,b\in A\;{\rm or}\;\mathbb{A}$.

Secondly, we note that $(a\#1)(1\#b)=a\#b$ and $(1\#a)(b\#1)=\{a.b\}\#1+(-1)^{|a||b|}b\#a$ in $A\#\mathbf U(A)$ for any homogenous elements $a,b\in A\;{\rm or}\;\mathbb{A}$. Hence we have
\begin{align*}
\alpha(h_{ab})&=\overline{1\#(ab)}=\overline{a\#b+(-1)^{|a||b|}b\#a}=\overline{(a\#1)(1\#b)+(-1)^{|a||b|}(b\#1)(a\#1)}=\alpha(m_ah_b+(-1)^{|a||b|}m_bh_a),\\
\alpha(m_{\{a,b\}})&=\overline{\{a,b\}\otimes 1}=\overline{(1\#a)(b\#1)-(-1)^{|a||b|}b\#a}=\alpha(h_am_b-(-1)^{|a||b|}m_bh_a),
\end{align*}
for all homogenous elements $a,b\in A\;{\rm or}\;\mathbb{A}$.

Finally, it is clear that $$d(\alpha(m_a))=d(a\#1)=\partial_A(a)\#1=\alpha(m_{\partial_A(a)})=\alpha(\partial(m_a))$$ and $$d(\alpha(h_a))=d(1\#a)=1\#\partial_A(a)=\alpha(h_{\partial_A(a)})=\alpha(\partial(h_a))$$ for all homogenous elements $a\in A\;{\rm or}\;\mathbb{A}$. This completes the proof.
\end{proof}

Next, we want to construct an inverse differential graded algebra map of $\alpha$. Note that in $A^e$, we have $h_{\{a,b\}}=h_ah_b-(-1)^{|a||b|}h_bh_a$ for all homogenous elements $a,b\in A$. Hence $h: (A,\{,\})\to (A^e)_L$ is a differential graded Lie algebra map. Thus there is a unique differential graded algebra map $\widetilde{h}: \mathbf U(A)\to A^e$ given by $\widetilde{h}(b_1\otimes \cdots \otimes b_s)=h_{b_1}\cdots h_{bs}$ for all $b_i\in A$. We define the map $\beta: A\#\mathbf{U}(A)\to A^e$ by
\begin{align*}
\beta(a\#f)=m_a\widetilde{h}(f)
\end{align*}
for all $a\in A$ and $f\in \mathbf U(A)$.

\begin{lemma}
The map $\beta: A\#\mathbf U(A)\to A^e$ is a well-defined differential graded algebra map.
\end{lemma}
\begin{proof}
Note that $\mathbf U(A)$ has a filtration $F_i$ spanned by $\{b_1\otimes \cdots \otimes b_s|b_j\in \mathbb{A},s\le i\}$. We want to show that
\begin{align*}
\beta\left((a\#f)(b\#g)\right)=\beta(a\#f)\beta(b\#g),
\end{align*}
for all $a,b\in A$ and $f,g\in \mathbf U(A)$. We prove the above identity by induction on the integer $n$ where $f\in F_n\setminus F_{n-1}$. When $n=0$, we can take $f=1_{\mathbf U(A)}$. Then
\begin{align*}
\beta\left((a\#1)(b\#g)\right)=\beta(ab\#g)=m_{ab}\widetilde{h}(g)=m_a(m_b\widetilde{h}(g))=\beta(a\#1)\beta(b\#g).
\end{align*}
Suppose that the above identity holds for any $f\in F_n\setminus F_{n-1}$. Now assume that $f\in F_{n+1}\setminus F_{n}$. Without loss of generality, we can further assume that $f$ is homogenous, and write $f=h\otimes c$ for some $h\in F_n\setminus F_{n-1}$ and $c\in \mathbb{A}$. Then we have
\begin{eqnarray*}
\beta\left((a\#f)(b\#g)\right)&=&\beta\left((a\#(h\otimes c))(b\#g)\right)\\
&=&\beta\left((a\#h)(1\#c)(b\#g)\right)\\
&=&\beta\left((a\#h)(\{c,b\}\#g+(-1)^{|c||b|}b\#(c\otimes g)\right)\\
&=&\beta(a\#h)\beta\left(\{c,b\}\#g+(-1)^{|c||b|}b\#(c\otimes g)\right)\\
&=&\beta(a\#h)\left(m_{\{c,b\}}\widetilde h(g)+(-1)^{|c||b|}m_b\widetilde h(c\otimes g)\right)\\
&=&\beta(a\#h)\left(m_{\{c,b\}}+(-1)^{|c||b|}m_bh_c\right)\widetilde h(g)\\
&=&\beta(a\#h)h_cm_b\widetilde h(g)\\
&=&\beta(a\#h)\beta(1\#c)\beta(b\#g)\\
&=&\beta(a\#(h\otimes c))\beta(b\#g)\\
&=&\beta(a\#f)\beta(b\#g).
\end{eqnarray*}
Therefore the induction step completes. Moreover, it is clear that
\begin{align*}
\beta(1_{A\#\mathbf U(A)})=\beta(1_A\#1_{\mathbf U(A)})=m_1\widetilde h(1_{\mathbf U(A)})=1_{A^e}.
\end{align*}
So $\beta$ is a graded algebra map. Note also that $\beta$ is compatible with the differentials, which follows from
\begin{eqnarray*}
\partial(\beta(a\#f))&=&\partial(m_a\widetilde h(f))\\
&=&\partial(m_a)\widetilde h(f)+(-1)^{|a|}m_a\partial(\widetilde h(f))\\
&=&m_{\partial_A(a)}\widetilde h(f)+(-1)^{|a|}m_a\widetilde h(\partial_{\mathbf U(A)} (f))\\
&=&\beta(\partial_A(a)\#f+(-1)^{|a|}a\#\partial_{\mathbf U(A)}(f))\\
&=&\beta(d(a\#f)),
\end{eqnarray*}
for all homogenous elements $a\in A$ and $f\in \mathbf U(A)$.
\end{proof}

Now we can state and prove our main result.
\begin{thm}\label{same}
Let $A\in\mathbf{DGPA}$. Then $A^e\cong A^E$ as differential graded algebras.
\end{thm}
\begin{proof}
It is clear that
\begin{align*}
\beta(1\#(ab))=\widetilde h(ab)=h_{ab}=m_ah_b+(-1)^{|a||b|}m_bh_a=\beta(a\#b+(-)^{|a||b|}b\#a),
\end{align*}
for all homogenous elements $a,b\in A\;{\rm or}\;\mathbb{A}$. Hence, the differential graded algebra map $\beta: A\#\mathbf U(A)\to A^e$ factors through $A^E$ as a differential graded algebra map, which is still denoted by $\beta$. Then we see that $\beta\alpha(m_a)=m_a$ and $\beta\alpha(h_a)=h_a$. Since $A^e$ is generated by $m_A$ and $h_A$, we have $\beta\alpha=1_{A^e}$. On the other hand, $\alpha\beta(a\#1)=a\#1$ and $\alpha\beta(1\#b)=1\#b$ for all $a\in A$ and $b\in\mathbb{A}$. By the same reason, we have $\alpha\beta=1_{A^E}$. This completes the proof.
\end{proof}

\begin{remark}
By Theorem \ref{same}, we can use $A^e$ and $A^E$ freely. In particular, the construction of $A^e$ is convenient to study the PBW-basis for the universal enveloping algebra; and the construction of $A^E$ is helpful to discuss the deformation and cohomology theory for $A$.
\end{remark}

\subsection{The property of $A^e$} We will study the relationship between $\mathbf{ DGP}(A)$ and $\mathbf{ DG}(A^e)$ in this subsection. The following is our main result of this subsection:
\begin{thm}\label{equivalence}
Let $A\in\mathbf{DGPA}$. Then $\mathbf{DGP}(A)\cong \mathbf{DG}(A^e)$.
\end{thm}

In order to prove Theorem \ref{equivalence}, we begin with some lemmas.
\begin{lemma}\label{lem2}
Let $(A,\{,\}_A,d)\in\mathbf{DGPA}$ and $(M,\{,\}_M,\partial)\in\mathbf{DGP}(A)$. Then $M\in\mathbf{DG}(A^e)$.
\end{lemma}
\begin{proof}
For any homogeneous elements $a\in A$ and $x\in M$, define an $A^e$-module action as follows:
$$m_a\cdot x:=ax,\;\;\;h_a\cdot x:=\{a,x\}_M.$$

Now we claim that the action we defined above preserves the relations of $A^e$. In fact, for any homogeneous elements $a,b\in A$ and $x\in M$, we have
\begin{eqnarray*}
m_{ab}\cdot x&=&(ab)x\\
&=&a(bx)\\
&=&(m_am_b)\cdot x,
\end{eqnarray*}
\begin{eqnarray*}
h_{ab}\cdot x&=&\{ab,x\}_M\\
&=&a\{b,x\}_M+(-1)^{|a||b|}b\{a,x\}_M\\
&=&(m_ah_b+(-1)^{|a||b|}m_bh_a)\cdot x,
\end{eqnarray*}
\begin{eqnarray*}
m_{\{a,b\}_A}\cdot x&=&\{a,b\}_Ax\\
&=&\{a,bx\}_M-(-1)^{|a||b|}b\{a,x\}_M\\
&=&(h_am_b-(-1)^{|a||b|}m_bh_a)\cdot x
\end{eqnarray*}
and
\begin{eqnarray*}
h_{\{a,b\}_A}\cdot x&=&\{\{a,b\}_A,x\}_M\\
&=&\{a,\{b,x\}_M\}_M-(-1)^{|a||b|}\{b,\{a,x\}_M\}_M\\
&=&(h_ah_b-(-1)^{|a||b|}h_bh_a)\cdot x,
\end{eqnarray*}
which finish the proof of the claim and we have that $M$ is a graded $A^e$ module. To complete the proof, we need to check that $\partial$ is compatible with the action ``$\cdot$'', which follows from
\begin{eqnarray*}
\partial(m_a\cdot x)&=&\partial(ax)\\
&=&d(a)x+(-1)^{|a|}a\partial(x)\\
&=&m_{d(a)}\cdot x+(-1)^{|a|}m_a\cdot\partial(x)\\
&=&\partial(m_a)\cdot x+(-1)^{|a|}m_a\cdot\partial(x)
\end{eqnarray*}
and
\begin{eqnarray*}
\partial(h_a\cdot x)&=&\partial(\{a,x\}_M)\\
&=&\{d(a),x\}_M+(-1)^{|a|}\{a,\partial(x)\}_M\\
&=&h_{d(a)}\cdot x+(-1)^{|a|}h_a\cdot \partial(x)\\
&=&\partial(h_a)\cdot x+(-1)^{|a|}h_a\cdot\partial(x).
\end{eqnarray*}
\end{proof}

\begin{lemma}\label{lem3}
Let $(A,\{,\}_A,d)\in\mathbf{DGPA}$ and $(M,\partial)\in\mathbf{DG}(A^e)$. Then $M\in\mathbf{DGP}(A)$.
\end{lemma}
\begin{proof}
Let $a\in A$ and $x\in M$ be any homogeneous elements. We define the left graded $A$-module structure and left graded Poisson module structure as follows:
$$a\ast x:=m_ax,\;\;\;\{a,x\}_M:=h_ax.$$
Now we claim that $(M,\ast,\{,\}_M,\partial)$ is a left differential $\mathbb{Z}$-graded Poisson module over $A$.

In fact, it is trivial that $\partial: M\to M$ is a $\Bbbk$-linear map of degree 1 such that $\partial^2=0$. Note that for any homogeneous elements $a,b\in A$ and $x\in M$, we have
\begin{eqnarray*}
\partial(a\ast x)&=&\partial(m_ax)\\
&=&\partial(m_a)x+(-1)^{|a|}m_a\partial(x)\\
&=&m_{d(a)}x+(-1)^{|a|}m_a\partial(x)\\
&=&d(a)\ast x+(-1)^{|a|}a\ast \partial(x),
\end{eqnarray*}
 \begin{eqnarray*}
 \{a,b\ast x\}_M&=&\{a,m_bx\}_M\\
 &=&h_am_bx\\
 &=&(m_{\{a,b\}_A}+(-1)^{|a||b|}m_bh_a)x\\
 &=&\{a,b\}_A\ast x+(-1)^{|a||b|}b\ast \{a,x\}_M,
 \end{eqnarray*}
\begin{eqnarray*}
 \{ab,x\}_M&=&h_{ab}x\\
 &=&(m_ah_b+(-1)^{|a||b|}m_bh_a)x\\
 &=&a\ast\{b,x\}_M+(-1)^{|a||b|}b\ast \{a,x\}_M,
 \end{eqnarray*}
 \begin{eqnarray*}
 \{a,\{b,x\}_M\}_M&=&h_ah_bx\\
 &=&(h_{\{a,b\}}+(-1)^{|a||b|}h_bh_a)x\\
 &=&\{\{a,b\}_M, x\}+(-1)^{|a||b|}\{b,\{a,x\}_M\}_M
 \end{eqnarray*}
and
\begin{eqnarray*}
 \partial(\{a,x\}_M)&=&\partial(h_ax)\\
 &=&\partial(h_a)x+(-1)^{|a|}h_a\partial(x)\\
 &=&h_{d(a)}x+(-1)^{|a|}h_a\partial(x)\\
 &=&\{d(a),x\}_M+(-1)^{|a|}\{a,\partial(x)\}_M.
 \end{eqnarray*}
 Therefore, we finish the proof.
 \end{proof}

From Lemmas \ref{lem2} and \ref{lem3}, we can see that there is a bijective correspondence between the objects of the category $\mathbf{ DGP}(A)$ and those of the category $\mathbf{ DG}(A^e)$. Further, the next result shows that there is also a bijective correspondence between the Hom-sets of the category $\mathbf{ DGP}(A)$ and  those of the category $\mathbf{ DG}(A^e)$, where $A\in\mathbf{DGPA}$.

\begin{lemma}\label{lem4}
We have the following observations.
\begin{enumerate}
  \item Let $A\in\mathbf{DGPA}$ and $M,N\in \mathbf{ DGP}(A)$. If $f:M\to N$ is a left differential graded Poisson $A$-module map, then it induces a differential graded $A^e$-module map $\widetilde{f}\;(=f)$.
  \item Let $A\in\mathbf{DGPA}$ and $M,N\in \mathbf{ DG}(A^e)$. If $g:M\to N$ is a left differential graded $A^e$-module map, then it induces a differential graded Poisson $A$-module map $\widehat{g}\;(=g)$.
\end{enumerate}
\end{lemma}

\begin{proof}
(1) It follows from
$$\widetilde{f}(m_a\cdot x)=f(ax)=af(x)=m_a\cdot f(x)=m_a\cdot \widetilde{f}(x)$$and
$$\widetilde{f}(h_a\cdot x)=\widetilde{f}(\{a,x\})=f(\{a,x\})=\{a,f(x)\}=h_a\cdot f(x)=h_a\cdot \widetilde{f}(x)$$ for any homogeneous elements $m_a,h_a\in A^e$ and $x\in M$, since $f$ is a left differential graded Poisson $A$-module map.

(2) It follows from
$$\widehat{g}(a\ast x)=g(m_ax)=m_ag(x)=a\ast g(x)=a\ast \widehat{g}(x)$$and
$$\widehat{g}(\{a,x\})=g(h_ax)=h_ag(x)=\{a,g(x)\}=\{a,\widehat{g}(x)\}$$for any homogeneous elements $a\in A$ and $x\in M$, since $g$ is a left differential graded $A^e$-module map.
\end{proof}

Now we are ready to prove Theorem \ref{equivalence}.
\begin{proof}[Proof of Theorem \ref{equivalence}]
We will prove that $A^e$ we constructed is the differential graded algebra we need.
Let $F: \mathbf{ DGP}(A)\to\mathbf{ DG}(A^e)$ be a map sending each left differential $\mathbb{Z}$-graded Poisson module $_AM$ to the differential $\mathbb{Z}$-graded module $_{A^e}M$, and sending a left differential graded Poisson module map $f:_AM\to _AN$ to the differential graded $A^e$-module map $\widetilde{f}:_{A^e}M\to_{A^e}N$. Let $G: \mathbf{ DG}(A^e)\to\mathbf{ DGP}(A)$ be a map sending each left differential $\mathbb{Z}$-graded $A^e$-module $_{A^e}M$ to the differential  $\mathbb{Z}$-graded Poisson $A$-module $_{A}M$, and sending a left differential graded $A^e$-module map $g:_{A^e}M\to _{A^e}N$ to the differential graded Poisson $A$-module map $\widehat{g}:_{A}M\to_{A}N$. Now it is easy to check that $F$ and $G$ are two covariant functors such that
$$GF=1_{\mathbf{ DGP}(A)}\;\;\;{\rm and}\;\;\;FG=1_{\mathbf{ DG}(A^e)},$$
which completes the proof.
\end{proof}

\subsection{Universal enveloping algebras of differential graded Poisson algebras}
In this subsection, first we will give the definition of {\it universal enveloping algebra} of a differential $\mathbb{Z}$-graded Poisson algebra, then we prove that the differential $\mathbb{Z}$-graded algebra $A^e$ is the universal enveloping algebra of the differential $\mathbb{Z}$-graded Poisson algebra $A$. In particular, we prove that ``$e$'' is a covariant functor from $\mathbf{DGPA}$ to $\mathbf{DGA}$ and $(\mathbb{S}L)^e\cong \mathbf{U}(L\rtimes L)$, where $L$ is a differential graded Lie algebra, $\mathbb{S}L$ is the graded symmetric algebra of $L$ and $\mathbf{U}(L\rtimes L)$ is the universal enveloping algebra of $L\rtimes L$, the semidirect product of $L$.
\begin{defn}\label{univ}
Let $(A,\{,\},d)\in\mathbf{DGPA}$ and $(A^{ue},\partial)\in\mathbf{DGA}$. We call $(A^{ue},\partial)$ is a {\it universal enveloping algebra} of $A$ if there exist a differential $\mathbb{Z}$-graded algebra map $\alpha: (A,d)\to (A^{ue},\partial)$ and a differential $\mathbb{Z}$-graded Lie algebra map $\beta:(A,\{,\},d)\to (A^{ue},[,],\partial)$ satisfying
\begin{align*}
\alpha(\{a,b\})=\beta(a)\alpha(b)-(-1)^{|a||b|}\alpha(b)\beta(a),\\
\beta(ab)=\alpha(a)\beta(b)+(-1)^{|a||b|}\alpha(b)\beta(a),
\end{align*}
for any homogeneous elements $a,b\in A$, such that for any $(D,\delta)\in\mathbf{DGA}$ with a differential $\mathbb{Z}$-graded algebra map $f: (A,d)\to (D,\delta)$ and a differential $\mathbb{Z}$-graded Lie algebra map $g:(A,\{,\},d)\to(D,[,],\delta)$ satisfying
\begin{align*}
f(\{a,b\})=g(a)f(b)-(-1)^{|a||b|}f(b)g(a),\\
g(ab)=f(a)g(b)+(-1)^{|a||b|}f(b)g(a),
\end{align*}
for any homogeneous elements $a,b\in A$, then there exists a unique differential $\mathbb{Z}$-graded algebra map $\phi: (A^{ue},\partial)\to(D,\delta)$, making this diagram
\[
\xymatrix{
(A,\{,\},d)\ar[rr]^-{\alpha}_-{\beta}\ar[dr]^-{f}_-{g} && (A^{ue},\partial)\ar@{-->}[dl]^-{\exists ! \phi}\\
& (D,\delta)&
}
\]
``bi-commute'', i.e., $\phi\alpha=f$ and $\phi\beta=g$.
\end{defn}

\begin{remark}\label{u}
From the Definition \ref{univ}, one can see that $A^{ue}$ admits certain ``universal property''. But note that it is different from the classical universal property---there does not exist an obvious adjoint functor. However, maybe one can understand the ``universal property'' of $A^{ue}$ by the following two observations:
\begin{itemize}
  \item Let $\mathbf{DGLA}$ denotes the category of differential graded Lie algebras and $D_L\in\mathbf{DGLA}$ with the ``standard graded Lie bracket'': $[x,y]=xy-(-1)^{|x||y|}yx$ for any homogeneous elements $x,y\in D$. We define $\widetilde{\Hom}(A,D)$ as a subset of $\Hom_{\mathbf{DGA}}(A,D)\times\Hom_{\mathbf{DGLA}}(A,D_L)$ with the property:  For each $\theta\in \widetilde{\Hom}(A,D)$, then $\theta=(f,g)$ such that $f\in\Hom_{\mathbf{DGA}}(A,D)$ and $g\in\Hom_{\mathbf{DGLA}}(A,D_L)$ such that
\begin{align*}
f(\{a,b\})=g(a)f(b)-(-1)^{|a||b|}f(b)g(a),\\
g(ab)=f(a)g(b)+(-1)^{|a||b|}f(b)g(a),
\end{align*}
for any homogeneous elements $a,b\in A$. Then we have
\[\Hom_{\mathbf{DGA}}(A^{ue}, D)\cong\widetilde{\Hom}(A,D).\]
\item For any given $P\in\mathbf{DGPA}$, we can define a category $\mathcal{P}$ as follows: The objects of $\mathcal{P}$ are a class of triples $(A,\alpha,\beta)$, where $A\in\mathbf{DGA}$, $\alpha: P\to A$ a differential graded algebra map and $\beta: P\to A$ a differential graded Lie algebra map such that
\begin{align*}
\alpha(\{a,b\})=\beta(a)\alpha(b)-(-1)^{|a||b|}\alpha(b)\beta(a),\\
\beta(ab)=\alpha(a)\beta(b)+(-1)^{|a||b|}\alpha(b)\beta(a),
\end{align*}
for any homogeneous elements $a,b\in A$. For any $(A,\alpha,\beta), (B,f,g)\in\mathcal{P}$, then the morphisms from $(A,\alpha,\beta)$ to $(B,f,g)$ are the differential graded algebra maps $\phi: A\to B$ such that the diagram
\[
\xymatrix{
P\ar[rr]^-{\alpha}_-{\beta}\ar[dr]^-{f}_-{g} && A\ar[dl]^-{\phi}\\
& B&
}
\]
``bi-commutes'', i.e., $\phi\alpha=f$ and $\phi\beta=g$. Then the universal enveloping algebra of $P$ is the initial object of the category $\mathcal{P}$.
\end{itemize}
\end{remark}

\begin{prop}\label{unique}
Let $(A,\{,\},d)\in\mathbf{DGPA}$ and $(A^{ue},\partial)$ be the universal enveloping algebra of $A$. Then $(A^{ue},\partial)$ is unique up to isomorphisms.
\end{prop}
\begin{proof}
Let $(\widehat{A}^{ue},\widehat{\partial})$ be another universal enveloping algebra of $A$. Then there exists a differential graded algebra map $\widehat{\alpha}: (A,d)\to (\widehat{A}^{ue},\widehat{\partial})$ and a differential graded Lie algebra map $\widehat{\beta}: (A\{,\},d)\to (\widehat{A}^{ue}, [,], \widehat{\partial})$ such that
\begin{align*}
\widehat{\alpha}_{\{a,b\}}=\widehat{\beta}_a\widehat{\alpha}_b-(-1)^{|a||b|}\widehat{\alpha}_b\widehat{\beta}_a,\\
\widehat{\beta}_{ab}=\widehat{\alpha}_a\widehat{\beta}_b+(-1)^{|a||b|}\widehat{\alpha}_b\widehat{\beta}_a,
\end{align*}
for any homogeneous elements $a,b\in A$.
Note that both $A^{ue}$ and $\widehat{A}^{ue}$ are universal enveloping algebras of $A$, by Definition \ref{univ}, there exist unique differential graded algebra maps $\phi: (A^{ue}, \partial)\to (\widehat{A}^{ue}, \widehat{\partial})$ and $\widehat{\phi}: (\widehat{A}^{ue}, \widehat{\partial})\to (A^{ue}, \partial)$, such that the following two diagrams

$$
\xymatrix{
(A,\{,\},d)\ar[rr]^-{\alpha}_-{\beta}\ar[dr]^-{\widehat{\alpha}}_-{\widehat{\beta}} && (A^{ue},\partial)\ar@{-->}[dl]^-{\exists ! \phi}\\
& (\widehat{A}^{ue},\widehat{\partial})&
},
\xymatrix{
(A,\{,\},d)\ar[rr]^-{\widehat{\alpha}}_-{\widehat{\beta}}\ar[dr]^-{\alpha}_-{\beta} && (\widehat{A}^{ue},\widehat{\partial})\ar@{-->}[dl]^-{\exists ! \widehat{\phi}}\\
& (A^{ue},\partial)&
},$$
``bi-commute''. In particular, we have the following two ``bi-commutative'' diagrams
$$
\xymatrix{
(A,\{,\},d)\ar[rr]^-{\alpha}_-{\beta}\ar[dr]^-{\alpha}_-{\beta} && (A^{ue},\partial)\ar@{-->}[dl]^-{\exists  \widehat{\phi}\phi\;{\rm and}\; 1_{A^{ue}}}\\
& (A^{ue},\partial)&
},
\xymatrix{
(A,\{,\},d)\ar[rr]^-{\widehat{\alpha}}_-{\widehat{\beta}}\ar[dr]^-{\widehat{\alpha}}_-{\widehat{\beta}} && (\widehat{A}^{ue},\widehat{\partial})\ar@{-->}[dl]^-{\exists  \phi\widehat{\phi}\;{\rm and}\; 1_{\widehat{A}^{ue}}}\\
& (\widehat{A}^{ue},\widehat{\partial})&
}.$$
By the uniqueness of the differential graded algebra maps in the above two bi-commutative diagrams, we have $\widehat{\phi}\phi=1_{A^{ue}}$ and $\phi\widehat{\phi}=1_{\widehat{A}^{ue}}$, as required.

\end{proof}

\begin{theorem}\label{universalproperty}
Let $(A,\{,\},d)\in\mathbf{DGPA}$. Then $A^e$ is the universal enveloping algebra of $A$, which is unique up to isomorphisms.
\end{theorem}

\begin{proof}
By Proposition \ref{unique}, it suffices to prove that $A^e$ is the universal enveloping algebra of $A$. From the construction of $A^e$, there are two $\Bbbk$-linear maps $m: A\to A^e$ sending each element $a\in A$ to $m_a$ and $h: A\to A^e$ sending each element $a\in A$ to $h_a$. In fact, $m$ is the required differential graded algebra map and $h$ is the required differential graded Lie algebra map since we have
\begin{align*}
m_{ab}=m_am_b,\\
h_{ab}=m_ah_b+(-1)^{|a||b|}m_bh_a,\\
m_{\{a,b\}}=h_am_b-(-1)^{|a||b|}m_bh_a,\\
h_{\{a,b\}}=h_ah_b-(-1)^{|a||b|}h_bh_a,\\
m_1=1,\\
\partial(m_a)=m_{d(a)},\\
\partial(h_a)=h_{d(a)},
\end{align*}
for any homogeneous elements $a,b\in A$.

Now let $(D,\delta)\in\mathbf{DGA}$ with a differential $\mathbb{Z}$-graded algebra map $f: (A,d)\to (D,\delta)$ and a differential $\mathbb{Z}$-graded Lie algebra map $g:(A,\{,\},d)\to(D,[,],\delta)$ satisfying
\begin{align*}
f(\{a,b\})=g(a)f(b)-(-1)^{|a||b|}f(b)g(a),\\
g(ab)=f(a)g(b)+(-1)^{|a||b|}f(b)g(a),
\end{align*}
for any homogeneous elements $a,b\in A$. Define a $\Bbbk$-linear map $\phi: (A^e,\partial)\to(D,\delta)$ by
$$\phi(m_a):=f(a)\;\;\;{\rm and}\;\;\;\phi(h_a):=g(a).$$
We claim that $\phi: A^e\to D$ is a $\mathbb{Z}$-graded algebra map. In fact, it suffices to check
\begin{align*}
\phi(m_{ab})=\phi(m_am_b),\\
\phi(h_{ab})=\phi(m_ah_b+(-1)^{|a||b|}m_bh_a),\\
\phi(m_{\{a,b\}})=\phi(h_am_b-(-1)^{|a||b|}m_bh_a),\\
\phi(h_{\{a,b\}})=\phi(h_ah_b-(-1)^{|a||b|}h_bh_a),
\end{align*}for any homogeneous elements $a,b\in A$,
which are true from
\begin{eqnarray*}
\phi(m_{ab})=f(ab)=f(a)f(b)=\phi(m_a)\phi(m_b),
\end{eqnarray*}
\begin{eqnarray*}
\phi(h_{ab})&=&g(ab)\\
&=&f(a)g(b)+(-1)^{|a||b|}f(b)g(a)\\
&=&\phi(m_a)\phi(h_b)+(-1)^{|a||b|}\phi(m_b)\phi(h_a)\\
&=&\phi(m_ah_b+(-1)^{|a||b|}m_bh_a),
\end{eqnarray*}
\begin{eqnarray*}
\phi(m_{\{a,b\}})&=&f({\{a,b\}})\\
&=&g(a)f(b)-(-1)^{|a||b|}f(b)g(a)\\
&=&\phi(h_a)\phi(m_b)-(-1)^{|a||b|}\phi(m_b)\phi(h_a)\\
&=&\phi(h_am_b-(-1)^{|a||b|}m_bh_a)
\end{eqnarray*}
and
\begin{eqnarray*}
\phi(h_{\{a,b\}})&=&g(\{a,b\})\\
&=&[g(a),g(b)]\\
&=&g(a)g(b)-(-1)^{|a||b|}g(b)g(a)\\
&=&\phi(h_a)\phi(h_b)-(-1)^{|a||b|}\phi(h_b)\phi(h_a)\\
&=&\phi(h_ah_b-(-1)^{|a||b|}h_bh_a).
\end{eqnarray*}

Further, $\phi$ is unique such that $\phi m=f$ and $\phi h=g$ by its construction and it is
also a differential graded algebra map since for any homogeneous element $a\in A$, we have
$$\phi(\partial(m_a))=\phi(m_{d(a)})=f(d(a))=\delta(f(a))=\delta(\phi(m_a))$$and
$$\phi(\partial(h_a))=\phi(h_{d(a)})=g(d(a))=\delta(g(a))=\delta(\phi(h_a)).$$
\end{proof}
\begin{remark}
From now on, we let $A^e$ denote the universal enveloping algebra of a differential graded Poisson algebra $A$ since Theorem \ref{universalproperty}.
\end{remark}

\begin{corollary}\label{e}
$e: \mathbf{ DGPA}\to \mathbf{ DGA}$ is a covariant functor.
\end{corollary}
\begin{proof}
For any $A\in \mathbf{ DGPA}$, define $e(A):=A^e$. By Lemma \ref{lem1}, $A^e\in\mathbf{ DGA}$. For any differential graded Poisson algebra map $f:A\to B$ in $\mathbf{ DGPA}$, note that $m^Bf:A\to B^e$ is a differential graded algebra map and $h^Bf: A\to B^e$ is a differential graded Lie algebra map satisfying
\begin{align*}
m^Bf(\{a,b\})=h^Bf(a)m^Bf(b)-(-1)^{|a||b|}m^Bf(b)h^Bf(a),\\
h^Bf(ab)=m^Bf(a)h^Bf(b)+(-1)^{|a||b|}m^Bf(b)h^Bf(a).
\end{align*}
Now by Theorem \ref{universalproperty}, $A^e$ is the universal enveloping algebra of $A$, then there exists a unique differential graded algebra map $f^e: A^e\to B^e$ making the following digram
\[
\xymatrix{
A\ar[rr]^-{m^A}_-{h^A}\ar[d]^-{f} && A^e\ar@{-->}[d]^-{\exists !  f^e}\\
 B\ar[rr]^-{m^B}_-{h^B}&&B^e}\]
``bi-commute''. Then we define $e(f):=f^e$. Further, it is clear that $e(1_A)=1_{A^e}$ for any $A\in\mathbf{ DGPA}$.

Now for any differential graded Poisson algebra maps $f: A\to B$ and $g:B\to C$, similarly, using the fact that $A^e$ and $B^e$ are the universal enveloping algebras of $A$ and $B$, respectively, we have the following ``bi-commutative'' diagram
\[
\xymatrix{
A\ar[rr]^-{m^A}_-{h^A}\ar[d]^-{f} && A^e\ar@{-->}[d]^-{\exists !  f^e}\\
 B\ar[rr]^-{m^B}_-{h^B}\ar[d]^-{g}&&B^e\ar@{-->}[d]^-{\exists !  g^e}\\
 C\ar[rr]^-{m^C}_-{h^C}&&C^e,}\]which implies that
 $$e(gf)=g^ef^e=e(g)e(f).$$
Therefore, $e: \mathbf{ DGPA}\to \mathbf{ DGA}$ is a covariant functor.
\end{proof}

\begin{remark}\label{rem1}
There is a natural question: For a given differential graded Poisson algebra $A$ in terms of generators and relations, how to compute $A^e$ explicitly.
In fact, let $V$ be a $\mathbb{Z}$-graded $\Bbbk$-vector space with a homogeneous $\Bbbk$-basis $\{x_{\alpha}: \alpha\in\Lambda\}$ and
\[R=\frac{T(V)}{(x_{\alpha}\otimes x_{\beta}-(-1)^{|x_{\alpha}||x_{\beta}|}x_{\beta}\otimes x_{\alpha}\;\big|\;\forall\alpha,\beta\in\Lambda)}.\]
Now suppose that $(R,d,\{,\})$ is a differential graded Poisson algebra and $I$ is a differential graded Poisson ideal of $R$. Put $A:=R/I.$ Then by Proposition \ref{2.7}, $A$ has a natural differential graded Poisson algebra structure induced from $R$.

Now for any $\alpha\in\Lambda$, we define a $\Bbbk$-linear map
\[\psi_{\alpha}:R\to R\]
by
\[
\psi_{\alpha}(x_{\beta})=\delta_{\alpha\beta}\;{\rm and}\;\psi_{\alpha}(ab)=a\psi_{\alpha}(b)+(-1)^{|a||b|}b\psi_{\alpha}(a)\]
for all $\alpha,\beta\in\Lambda$ and for all homogeneous elements $a,b\in R$. Note that these $\Bbbk$-linear maps $\{\psi_{\alpha}\}_{\alpha\in\Lambda}$ are well-defined since they preserve the relations of $R$.

Let $\{y_{\alpha}|\alpha\in\Lambda\}$ be another copy of the homogeneous basis of $V$ such that $|x_{\alpha}|=|y_{\alpha}|$ for any $\alpha\in\Lambda$. Let $F(R)=R\langle y_\alpha|\alpha\in\Lambda\rangle $ be the free $R$-algebra generated by $\{y_{\alpha}|\alpha\in\Lambda\}$. Now follow the idea of \cite{OPS} and by the same argument with some extra signs, we define a degree 0 $\Bbbk$-linear map $\Psi: R\to F(R)$ such that
$$\Psi(f):=\sum_{\alpha\in\Lambda}\psi_{\alpha}(f)y_{\alpha},$$ for any $f\in R$. Note that such $\Psi$ is well-defined since for any $f\in R$, there are only finite many $\alpha\in\Lambda$ such that $\psi_{\alpha}(f)$ is not zero.
Then we get a differential graded algebra
$${\small \mathcal{A}=\frac{F(R)}{(I,\Psi(I),y_{\alpha}x_{\beta}-(-1)^{|x_{\alpha}||x_{\beta}|}x_{\beta}y_{\alpha}-\{x_{\alpha},x_{\beta}\},y_{\alpha}y_{\beta}-(-1)^{|x_{\alpha}||x_{\beta}|}y_{\beta}y_{\alpha}-\Psi(\{x_{\alpha},x_{\beta}\})\;\Big|\;\alpha,\beta\in\Lambda)}},$$
where the differential $\partial$ of $\mathcal{A}$ is given by
\[\partial(x_{\alpha}):=d(x_{\alpha}),\;\;\;\\ \partial(y_{\alpha}):=\overline{\Psi(d(x_{\alpha}))}
\] and the graded Leibniz rule.  Note that there are two $\Bbbk$-linear maps $m,h: A\to \mathcal{A}$ given by
\[
m(f)=f,\; h(f)=\overline{\Psi(f)},
\]
for any $f\in A$. In fact, $(\mathcal{A},m,h)$ is the universal enveloping algebra of $A$. Note also that, if $I=0$, then $A^e$ has a PBW-basis. Concerning the length of the paper and the complicated computation, we will treat this in \cite{LWZ5}.
\end{remark}

To understand Remark \ref{rem1}, we consider the following example.
\begin{exa}
{\rm Let $A$ be the same differential graded Poisson algebra as defined in Example \ref{con}. Then by Remark \ref{rem1}, we have that $A^e$ is generated by $\{x_1,x_2,y_1,y_2\}$ subject to the following relations
\begin{align*}
\label{}
x_1^2, y_1^2, \\
[x_1,x_2], [x_1,y_1], [x_2,y_2], \\
[y_1,y_2]-p(x_2y_1+x_1y_2),
[y_1,x_2]-px_1x_2,
[y_2,x_1]+px_1x_2,
\end{align*}
and the differential $\partial$ is given by
\[\partial(x_1)=\lambda x_2, \partial(x_2)=\mu x_1x_2, \partial(y_1)=\lambda y_2, \partial(y_2)=\mu(x_2y_1+x_1y_2).\]

}
\end{exa}

Note that we already have the other two notions of ``universal enveloping algebra''---$R^{\mathcal{E}}:=R\otimes R^{op}$ for any ring $R$ and $\mathbf{ U}(L)$, the universal enveloping algebra of Lie algebra $L$. Concerning this, one can ask the following natural question:
\begin{itemize}
  \item How do those constructions interact with each other?
\end{itemize}
We will study the relation between $A^e$ and $\mathbf{U}(A)$ in the rest of this subsection for the case of $A=\mathbb{S}L$, the graded symmetric algebra of a differential graded Lie algebra $L$ (see Theorem \ref{semi}), and we will discuss the relationship between $A^e$ and $A^{\mathcal{E}}$ in the next subsection (see Corollary \ref{coop}).

\begin{prop}\label{sem}
Let $(L,[,]_L, d)$ be a differential graded Lie algebra. Then the graded semidirect product $L\rtimes L$ of $L$ is a differential graded Lie algebra with the bracket $\{,\}$ and differential $\partial$ defined by
\begin{align*}
\{(x_1,x_2),(y_1,y_2)\}:=([x_2,y_1]_L-(-1)^{|x||y|}[y_2,x_1]_L,[x_2,y_2]_L),\\
\partial(x_1,x_2):=(d(x_1),d(x_2)),
\end{align*}
where $(x_1,x_2)$ and $(y_1,y_2)$ are homogeneous elements of $L\rtimes L$ and $|x|:=|(x_1,x_2)|=|x_1|=|x_2|$, $|y|:=|(y_1,y_2)|=|y_1|=|y_2|$.
\end{prop}
\begin{proof}
Let $(x_1,x_2), (y_1,y_2), (z_1,z_2)\in L\rtimes L$ be any homogeneous elements.

The ``graded antisymmetry'' property follows from
\begin{eqnarray*}
\{(x_1,x_2),(y_1,y_2)\}&=&([x_2,y_1]_L-(-1)^{|x||y|}[y_2,x_1]_L,[x_2,y_2]_L)\\
&=&-(-1)^{|x||y|}([y_2,x_1]_L-(-1)^{|x||y|}[x_2,y_1]_L,[y_2,x_2]_L)\\
&=&-(-1)^{|x||y|}\{(y_1,y_2),(x_1,x_2)\}.
\end{eqnarray*}

The ``graded Leibniz rule for $\{,\}$'' follows from
\begin{eqnarray*}
&&\partial(\{(x_1,x_2),(y_1,y_2)\})\\
&=&\partial([x_2,y_1]_L-(-1)^{|x||y|}[y_2,x_1]_L,[x_2,y_2]_L)\\
&=&([dx_2,y_1]_L+(-1)^{|x|}[x_2,dy_1]_L-(-1)^{|x||y|}[dy_2,x_1]_L-(-1)^{|y|(|x|+1)}[y_2,dx_1]_L, [dx_2,y_2]_L+(-1)^{|x|}[x_2,dy_2]_L)
\end{eqnarray*}
and
\begin{eqnarray*}
&&\{\partial(x_1,x_2),(y_1,y_2)\}+(-1)^{|x|}\{(x_1,x_2),\partial(y_1,y_2)\}\\
&=&\{(dx_1,dx_2),(y_1,y_2)\}+(-1)^{|x|}\{(x_1,x_2),(dy_1,dy_2)\}\\
&=&([dx_2,y_1]_L-(-1)^{|y|(|x|+1)}[y_2,dx_1]_L,[dx_2,y_2]_L)+(-1)^{|x|}([x_2,dy_1]_L-(-1)^{|x|(|y|+1)}[dy_2,x_1]_L,[x_2,dy_2]_L)\\
&=&([dx_2,y_1]_L+(-1)^{|x|}[x_2,dy_1]_L-(-1)^{|x||y|}[dy_2,x_1]_L-(-1)^{|y|(|x|+1)}[y_2,dx_1]_L, [dx_2,y_2]_L+(-1)^{|x|}[x_2,dy_2]_L).
\end{eqnarray*}

For the ``graded Jacobi identity''  for $\{,\}$, first we have
\begin{eqnarray*}
&&\{(x_1,x_2),\{(y_1,y_2),(z_1,z_2)\}\}\\
&=&\{(x_1,x_2),([y_2,z_1]_L-(-1)^{|y||z|}[z_2,y_1]_L, [y_2,z_2]_L)\}\\
&=& ([x_2,[y_2,z_1]_L-(-1)^{|y||z|}[z_2,y_1]_L]_L-(-1)^{|x|(|y|+|z|)}[[y_2,z_2]_L,x_1]_L, [x_2,[y_2,z_2]_L]_L)\\
&=&([x_2,[y_2,z_1]_L]_L-(-1)^{|y||z|}[x_2,[z_2,y_1]_L]_L-(-1)^{|x|(|y|+|z|)}[[y_2,z_2]_L,x_1]_L, [x_2,[y_2,z_2]_L]_L),
\end{eqnarray*}
\begin{eqnarray*}
&&\{\{(x_1,x_2),(y_1,y_2)\},(z_1,z_2)\}\\
&=&\{([x_2, y_1]_L-(-1)^{|x||y|}[y_2,x_1]_L, [x_2,y_2]_L),(z_1,z_2)\}\\
&=&([[x_2,y_2]_L,z_1]_L-(-1)^{(|x|+|y|)|z|}[z_2,[x_2,y_1]_L-(-1)^{|x||y|}[y_2,x_1]_L]_L, [[x_2,y_2]_L,z_2]_L)\\
&=&([[x_2,y_2]_L,z_1]_L-(-1)^{(|x|+|y|)|z|}[z_2,[x_2,y_1]_L]_L+(-1)^{|x||z|+|y||z|+|x||y|}[z_2,[y_2,x_1]_L]_L, [[x_2,y_2]_L,z_2]_L)
\end{eqnarray*}
and
\begin{eqnarray*}
&&(-1)^{|x||y|}\{(y_1,y_2),\{(x_1,x_2),(z_1,z_2)\}\}\\
&=&(-1)^{|x||y|}\{(y_1,y_2),([x_2,z_1]_L-(-1)^{|x||z|}[z_2,x_1]_L,[x_2,z_2]_L)\}\\
&=&(-1)^{|x||y|}([y_2,[x_2,z_1]_L-(-1)^{|x||z|}[z_2,x_1]_L]_L-(-1)^{|y|(|x|+|z|)}[[x_2,z_2]_L,y_1]_L, [y_2,[x_2,z_2]_L]_L)\\
&=&(-1)^{|x||y|}([y_2,[x_2,z_1]_L]_L-(-1)^{|x||z|}[y_2,[z_2,x_1]_L]_L-(-1)^{|y|(|x|+|z|)}[[x_2,z_2]_L,y_1]_L, [y_2,[x_2,z_2]_L]_L)\\
&=&((-1)^{|x||y|}[y_2,[x_2,z_1]_L,]_L-(-1)^{|x|(|y|+|z|)}[y_2,[z_2,x_1]_L]_L-(-1)^{|y||z|}[[x_2,z_2]_L,y_1]_L, (-1)^{|x||y|}[y_2,[x_2,z_2]_L]_L).
\end{eqnarray*}

Recall that $L$ is a graded Lie algebra, by the ``graded Jacobi identity'' for $[,]_L$, we have
\begin{align*}
[x_2,[y_2,z_2]_L]_L=[[x_2,y_2]_L,z_2]_L+(-1)^{|x||y|}[y_2,[x_2,z_2]_L]_L,\\
[x_2,[y_2,z_1]_L]_L=[[x_2,y_2]_L,z_1]_L+(-1)^{|x||y|}[y_2,[x_2,z_1]_L,]_L,\\
-(-1)^{|y||z|}[x_2,[z_2,y_1]_L]_L=-(-1)^{(|x|+|y|)|z|}[z_2,[x_2,y_1]_L]_L-(-1)^{|y||z|}[[x_2,z_2]_L,y_1]_L,\\
-(-1)^{|x|(|y|+|z|)}[[y_2,z_2]_L,x_1]_L=(-1)^{|x||z|+|y||z|+|x||y|}[z_2,[y_2,x_1]_L]_L-(-1)^{|x|(|y|+|z|)}[y_2,[z_2,x_1]_L]_L.
\end{align*}

Therefore, we are done.
\end{proof}

\begin{theorem}\label{semi}
Let $(L,[,]_L,d)$ be a differential graded Lie algebra, and $\mathbb{S}L$ be its graded symmetric algebra. Then $\mathbb{S}L$ is a differential graded Poisson algebra and
\[(\mathbb{S}L)^e=\mathbf{U}(L\rtimes L).\]as differential graded algebras.
\end{theorem}
\begin{proof}
By Proposition \ref{p1}, $\mathbb{S}L$ is a differential graded Poisson algebra. Now let $\{x_{\alpha}\}_{\alpha\in \Lambda}$ be a basis of $L$ consisting of homogenous elements. In the graded symmetric algebra $\mathbb{S}L$, we can define ``anti-differentials'' $\psi_\alpha$ associated to every basis $x_\alpha$ such that
\begin{align*}
\psi_\alpha(ab)=a\psi_\alpha(b)+(-1)^{|a||b|}b\psi_\alpha(a),\quad \psi_\alpha(x_\beta)=\delta_{\alpha\beta},
\end{align*}
for all homogenous elements $a,b\in \mathbb{S}L$. Then we get a $\Bbbk$-linear map $\Psi$ on $\mathbb{S}L$ given by
\begin{align*}
\Psi(a)=\sum_{\alpha\in\Lambda} \psi_\alpha(a)y_\alpha,
\end{align*}
for all $a\in \mathbb{S}L$. Note that for every $a\in \mathbb{S}L$, there are only finite many $\alpha\in\Lambda$ such that $\psi_\alpha(a)\neq 0$. So $\Psi$ is well-defined on $\mathbb{S}L$.

By Remark \ref{rem1} in the case of $I=0$, we see that the universal enveloping algebra $(\mathbb{S}L)^e$ is generated by $\{x_\alpha,y_\alpha\}_{\alpha\in\Lambda}$, subject to the following relations:
\begin{itemize}
\item $x_\alpha x_\beta-(-1)^{|x_\alpha||x_\beta|}x_\beta x_\alpha=0$;
\item $y_\alpha x_\beta-(-1)^{|y_\alpha||x_\beta|}x_\beta y_\alpha-\{x_\alpha,x_\beta\}=0$;
\item $y_\alpha y_\beta-(-1)^{|y_\alpha||y_\beta|}y_\beta y_\alpha-\Psi(\{x_\alpha,x_\beta\})=0$,
\end{itemize}
for all bases. If $\{x_\alpha,x_\beta\}=\sum_{\alpha\in\Lambda} f_{\alpha}x_\alpha$ for some coefficients $f_\alpha\in \Bbbk$, we have $\psi(\{x_\alpha,x_\beta\})=\sum_{\alpha\in\Lambda} f_{\alpha}y_\alpha$. Finally, it is routine to check that $\mathbf{U}(L\rtimes L)$ is isomorphic to $(\mathbb{S}L)^e$ by sending $(x_\alpha,y_\beta)$ to $x_\alpha+y_\beta$ for all the generators $(x_\alpha,y_\beta)\in L\rtimes L$.\end{proof}

\medskip
\section{Applications of the ``universal property''}
\subsection{On $(A^e)^{op}=(A^{op})^e$}
We begin with the following lemma.
\begin{lemma}\label{opop}
Let $A,B,C\in \mathbf{ DGPA}$. We have the following statements:
\begin{enumerate}
  \item If $f: A\to B$ is a differential graded (Lie) algebra map, then $f^{op}: A^{op}\to B^{op}$ defined by $f^{op}(a):=f(a)$ for any homogeneous element $a\in A$, is also a differential graded (Lie) algebra map. As a consequence, if $f$ is a differential graded Poisson algebra map, then so is $f^{op}$.

  \item $(A^{op})^{op}=A$ and $(f^{op})^{op}=f$, where $f$ is any differential graded Poisson algebra map in $\mathbf{ DGPA}$.

  \item If $f: A\to B$ and $g:B\to C$ be two differential graded (Lie) algebra maps, then $(gf)^{op}=g^{op}f^{op}: A^{op}\to C^{op}$ is also a differential graded (Lie) algebra map.
  \item $\Gamma$ is a commutative diagram in $\mathbf{ DGPA}$ if and only if $\Gamma^{op}$ is a commutative diagram in $\mathbf{ DGPA}$, where $\Gamma^{op}$ is obtained by replacing each map $f$ by $f^{op}$.
\end{enumerate}
\end{lemma}
\begin{proof}
It suffices to check (1) since (2)-(4) are immediate from the definition of ``$op$''. If $A,B\in \mathbf{ DGPA}$ and $f: A\to B$ is a differential graded algebra map, then for any homogeneous elements $a,b\in A$, then $f(ab)=f(a)f(b)$ and $fd^A=d^Bf$. Now for any homogeneous elements $x,y\in A^{op}$, by Example \ref{op} we have
$$f^{op}(xy)=f(xy)=f(x)f(y)=f^{op}(x)f^{op}(y)$$
 and
 $$d_{op}^Bf^{op}(x)=d^Bf(x)=fd^A(x)=f^{op}d^A_{op}(x),$$
which imply that $f^{op}: A^{op}\to B^{op}$ is a differential graded algebra map. If $f: A\to B$ is a differential graded Lie algebra map, then for any homogeneous elements $a,b\in A$, then $f(\{a,b\}_A)=\{f(a),f(b)\}_B$ and $fd^A=d^Bf$. Now for any homogeneous elements $x,y\in A^{op}$, by Example \ref{op} we have
$$f^{op}(\{x,y\}_{A^{op}})=-f\{x,y\}_A=-\{f(x),f(y)\}_B=\{f^{op}(x),f^{op}(y)\}_{B^{op}}$$
and
$$\;d_{op}^Bf^{op}(x)=d^Bf(x)=fd^A(x)=f^{op}d^A_{op}(x),$$
which imply that $f^{op}: A^{op}\to B^{op}$ is a differential graded Lie algebra map.
\end{proof}

\begin{theorem}\label{thop}
Let $A\in\mathbf{ DGPA}$. Then $(A^{op})^e=(A^e)^{op}$.
\end{theorem}
\begin{proof}
From the above we know that there are $\Bbbk$-linear maps $m,h:A\to A^e$ such that $m$ is a differential graded algebra map and $h$ is a differential graded Lie algebra map with the properties
\begin{align}
h_{ab}=m_ah_b+(-1)^{|a||b|}m_bh_a,\\
m_{\{a,b\}}=h_am_b-(-1)^{|a||b|}m_bh_a,
\end{align}
for any homogeneous elements $a,b\in A$. By the definition and Lemma \ref{opop}, we have the $\Bbbk$-linear maps $m^{op},h^{op}:A^{op}\to (A^e)^{op}$ such that $m^{op}$ is a differential graded algebra map and $h^{op}$ is a differential graded Lie algebra map.

Moreover, we claim that\begin{align}
h^{op}_{a\cdot_{op}b}=m^{op}_a\cdot_{op}h^{op}_b+(-1)^{|a||b|}m^{op}_b\cdot_{op}h^{op}_a,\\
m^{op}_{\{a,b\}_{op}}=h^{op}_a\cdot_{op}m^{op}_b-(-1)^{|a||b|}m^{op}_b\cdot_{op}h^{op}_a,
\end{align}
for any homogeneous elements $a,b\in A^{op}$.  In fact, we will check (4.3) first. By (4.2), we have
\begin{align}
h_bm_a-(-1)^{|a||b|}m_ah_b=m_{\{b,a\}},\\
h_am_b-(-1)^{|a||b|}m_bh_a=m_{\{a,b\}},
\end{align}
for any homogeneous elements $a,b\in A$. Now we combine (4.1), (4.5) and (4.6), we get
\begin{align}h_am_b+(-1)^{|a||b|}h_bm_a=h_{ab}\end{align}
for any homogeneous elements $a,b\in A$. Now (4.3) follows from by taking ``$^{op}$'' to both sides of (4.7). Now for (4.4), note that (4.4) is equivalent to
\begin{align*}-m_{\{a,b\}}=(-1)^{|a||b|}m_bh_a-h_am_b,\end{align*} for any homogeneous elements $a,b\in A$, which is the same as (4.2).

Now let $D^{op}$ be any differential $\mathbb{Z}$-graded algebra with a differential $\mathbb{Z}$-graded algebra map $f^{op}: A^{op}\to D^{op}$ and a differential $\mathbb{Z}$-graded Lie algebra map $g^{op}:A^{op}\to D^{op}$ satisfying
\begin{align*}
f^{op}(\{a,b\}_{op})=g^{op}(a)\cdot_{op}f^{op}(b)-(-1)^{|a||b|}f^{op}(b)\cdot_{op}g^{op}(a),\\
g^{op}(a\cdot_{op}b)=f^{op}(a)\cdot_{op}g^{op}(b)+(-1)^{|a||b|}f^{op}(b)\cdot_{op}g^{op}(a),
\end{align*}
for any homogeneous elements $a,b\in A^{op}$. Similar to the proof of the above claim, we have a differential graded algebra map $f: A\to D$ and a differential graded Lie algebra map $g: A\to D$ such that \begin{align*}
f(\{a,b\})=g(a)f(b)-(-1)^{|a||b|}f(b)g(a),\\
g(ab)=f(a)g(b)+(-1)^{|a||b|}f(b)g(a),
\end{align*}
for any homogeneous elements $a,b\in A$. Note that $A^e$ is the universal enveloping algebra of $A$, there exists a unique differential graded algebra map $\phi$, such that the following diagram
\[
\xymatrix{
A\ar[rr]^-{m}_-{h}\ar[dr]^-{f}_-{g} && A^{e}\ar@{-->}[dl]^-{\exists ! \phi}\\
& D&
}
\]
``bi-commutes''. By Lemma \ref{opop} again, there exists a unique differential graded algebra map $\phi^{op}$, such that the following diagram
\[
\xymatrix{
A^{op}\ar[rr]^-{m^{op}}_-{h^{op}}\ar[dr]^-{f^{op}}_-{g^{op}} && (A^{e})^{op}\ar@{-->}[dl]^-{\exists ! \phi^{op}}\\
& D^{op}&
}
\]
``bi-commutes''. Therefore, we have $(A^{op})^e=(A^e)^{op}$.
\end{proof}

\subsection{On $(A\otimes B)^e=A^e\otimes B^e$}
Now we turn to prove that $(A\otimes B)^e=A^e\otimes B^e$ for any $(A,\{,\}_A,d_A),(B,\{,\}_B,d_B)\in\mathbf{DGPA}$. By the definition of universal enveloping algebra of a differential graded Poisson algebra, we know that there are $\Bbbk$-linear maps $m^A, h^A, m^B$ and $h^B$ such that
$$m^A: (A,d_A)\to (A^e,\partial^A)\;\;\;{\rm and}\;\;\;m^B: (B,d_B)\to (B^e,\partial^B)$$
are differential graded algebra maps, and
$$h^A: (A,d_A,\{,\}_A)\to (A^e,\partial^A,[,])\;\;\;{\rm and}\;\;\;h^B: (B,d_B,\{,\}_B)\to (B^e,\partial^B,[,])$$
are differential graded Lie algebra maps, satisfying
\begin{align*}
h^A_{aa'}=m^A_ah^A_{a'}+(-1)^{|a||a'|}m^A_{a'}h^A_a,\;\;
m^A_{\{a,a'\}}=h^A_am^A_{a'}-(-1)^{|a||a'|}m^A_{a'}h^A_a,\\
h^B_{bb'}=m^B_bh^B_{b'}+(-1)^{|b||b'|}m^B_{b'}h^B_b,\;\;
m^B_{\{b,b'\}}=h^B_bm^B_{b'}-(-1)^{|b||b'|}m^B_{b'}h^B_b,
\end{align*}
for any homogeneous elements $a,a'\in A$ and $b,b'\in B$.

Retain the above notions and we begin with some lemmas.

\begin{lemma}\label{ot1}
 We have that
$$m:=m^A\otimes m^B: (A\otimes B, d_{A\otimes B}\to (A^e\otimes B^e,\partial_{A^e\otimes B^e})$$ is a differential graded algebra map.
\end{lemma}
\begin{proof}
Let $a,a'\in A$ and $b,b'\in B$ be any homogeneous elements. Then $m$ is a graded algebra map since
\begin{eqnarray*}
m((a\otimes b)(a'\otimes b'))&=&(-1)^{|a'||b|}(m^A\otimes m^B)(aa'\otimes bb')\\
&=&(-1)^{|a'||b|}m^A(aa')\otimes m^B(bb')\\
&=&(-1)^{|a'||b|}m^A_am^A_{a'}\otimes m^B_bm^B_{b'}\\
&=&(m^A_a\otimes m^B_b)(m^A_{a'}\otimes m^B_{b'})\\
&=&m(a\otimes b)m(a'\otimes b').
\end{eqnarray*}
The algebra map $m$ is compatible with the differentials since
\begin{eqnarray*}
md_{A\otimes B}(a\otimes b)&=&(m^A\otimes m^B)(d_A(a)\otimes b+(-1)^{|a|}a\otimes d_B(b))\\
&=&m^A_{d_A(a)}\otimes m^B_b+(-1)^{|a|}m^A_a\otimes m^B_{d_B(b)}\\
&=&\partial_{A^e}(m^A_a)\otimes m^B_b+(-1)^{|a|}m^A_a\otimes\partial_{B^e}(m^B_b)\\
&=&\partial_{A^e\otimes B^e}(m^A_a\otimes m^B_b)\\
&=&\partial_{A^e\otimes B^e}m(a\otimes b).
\end{eqnarray*}
\end{proof}

\begin{lemma}\label{ot2}
We have that
\[h:=m^A\otimes h^B+h^A\otimes m^B: (A\otimes B, \{,\}_{A\otimes B}, d_{A\otimes B})\to (A^e\otimes B^e, [,], \partial_{A^e\otimes B^e})\]
is a differential graded Lie algebra map.
\end{lemma}
\begin{proof}
For any homogeneous elements $a\otimes b, a'\otimes b'\in A\otimes B$, we have
\begin{eqnarray*}
&&h(\{a\otimes b,a'\otimes b'\}_{A\otimes B})\\
&=&(m^A\otimes h^B+h^A\otimes m^B)(\{a\otimes b,a'\otimes b'\}_{A\otimes B})\\
&=&(-1)^{|a'||b|}(m^A\otimes h^B+h^A\otimes m^B)(aa'\otimes \{b,b'\}_B+\{a,a'\}_A\otimes bb')\\
&=&(-1)^{|a'||b|}m^A_am^A_{a'}\otimes h^B_bh^B_{b'}-(-1)^{|a'||b|+|b||b'|}m^A_am^A_{a'}\otimes h^B_{b'}h^B_b+(-1)^{|a'||b|}h^A_am^A_{a'}\otimes m^B_bh^B_{b'}\\
&&+(-1)^{|a'||b|+|b||b'|}h^A_am^A_{a'}\otimes m^B_{b'}h^B_b-(-1)^{|a'||b|+|a||a'|}m^A_{a'}h^A_a\otimes m^B_bh^B_{b'}-(-1)^{|a'||b|+|a||a'|+|b||b'|}m^A_{a'}h^A_a\otimes m^B_{b'}h^B_b\\
&&+(-1)^{|a'||b|}m_a^Ah^A_{a'}\otimes h^B_bm^B_{b'}-(-1)^{|a'||b|+|b||b'|}m^A_ah^A_{a'}\otimes m^B_{b'}h^B_b+(-1)^{|a'||b|+|a||a'|}m^A_{a'}h^A_a\otimes h^B_bm^B_{b'}\\
&&-(-1)^{|a'||b|+|a||a'|+|b||b'|}m^A_{a'}h^A_{a}\otimes m^B_{b'}h^B_b+(-1)^{|a'||b|}h^A_ah^A_{a'}\otimes m^B_bm^B_{b'}-(-1)^{|a'||b|+|a||a'|}h^A_{a'}h^A_a\otimes m^B_bm^B_{b'}
\end{eqnarray*}
and
\begin{eqnarray*}
&&[h(a\otimes b),h(a'\otimes b')]\\
&=&[(m^A\otimes h^B+h^A\otimes m^B)(a\otimes b),(m^A\otimes h^B+h^A\otimes m^B)(a'\otimes b')]\\
&=&(m_a^A\otimes h_b^B+h_a^A\otimes m_b^B)(m_{a'}^A\otimes h_{b'}^B+h_{a'}^A\otimes m_{b'}^B)-(-1)^{(|a|+|b|)(|a'|+|b'|)}(m_{a'}^A\otimes h_{b'}^B+h_{a'}^A\otimes m_{b'}^B)(m_a^A\otimes h_b^B+h_a^A\otimes m_b^B)\\
&=&(-1)^{|a'||b|}m^A_am^A_{a'}\otimes h^B_bh^B_{b'}+(-1)^{|a'||b|}m^A_ah^A_{a'}\otimes h^B_bm^B_{b'}+(-1)^{|a'||b|}h^A_am^A_{a'}\otimes m^B_bh^B_{b'}+(-1)^{|a'||b|}h^A_ah^A_{a'}\otimes m^B_bm^B_{b'}\\
&&-(-1)^{|a'||b|+|a||a'|+|b||b'|}m^A_{a'}m^A_a\otimes h^B_{b'}h^B_b-(-1)^{|a'||b|+|a||a'|+|b||b'|}m^A_{a'}h^A_a\otimes h^B_{b'}m^B_b\\
&&-(-1)^{|a'||b|+|a||a'|+|b||b'|}h^A_{a'}m^A_a\otimes m^B_{b'}h^B_b-(-1)^{|a'||b|+|a||a'|+|b||b'|}h^A_{a'}h^A_a\otimes m^B_{b'}m^B_b.
\end{eqnarray*}

Note that there are six exactly same terms in $h(\{a\otimes b,a'\otimes b'\}_{A\otimes B})$ and $[h(a\otimes b),h(a'\otimes b')]$, thus in order to prove
$$h(\{a\otimes b,a'\otimes b'\}_{A\otimes B})=[h(a\otimes b),h(a'\otimes b')],$$it suffices to prove
\begin{eqnarray*}
&&-(-1)^{|a'||b|+|a||a'|+|b||b'|}m^A_{a'}h^A_a\otimes h^B_{b'}m^B_b-(-1)^{|a'||b|+|a||a'|+|b||b'|}h^A_{a'}m^A_a\otimes m^B_{b'}h^B_b\\
&=&(-1)^{|a'||b|+|b||b'|}h^A_am^A_{a'}\otimes m^B_{b'}h^B_b-(-1)^{|a'||b|+|a||a'|}m^A_{a'}h^A_a\otimes m^B_bh^B_{b'}-(-1)^{|a'||b|+|a||a'|+|b||b'|}m^A_{a'}h^A_{a}\otimes m^B_{b'}h^B_b\\
&&-(-1)^{|a'||b|+|b||b'|}m^A_ah^A_{a'}\otimes m^B_{b'}h^B_b+(-1)^{|a'||b|+|a||a'|}m^A_{a'}h^A_a\otimes h^B_bm^B_{b'}-(-1)^{|a'||b|+|a||a'|+|b||b'|}m^A_{a'}h^A_{a}\otimes m^B_{b'}h^B_b.
\end{eqnarray*}

Note that
\[
(-1)^{|a'||b|+|b||b'|}h^A_am^A_{a'}\otimes m^B_{b'}h^B_b=(-1)^{|a'||b|+|b||b'|}(m^A_{\{a,a'\}}+(-1)^{|a||a'|}m^A_{a'}h^A_{a})\otimes m^B_{b'}h^B_b\]
and
\[
(-1)^{|a'||b|+|a||a'|}m^A_{a'}h^A_a\otimes h^B_bm^B_{b'}=(-1)^{|a'||b|+|a||a'|}m^A_{a'}h^A_a\otimes ((-1)^{|b||b'|}m^B_{b'}h^B_b+m^B_{\{b,b'\}}),\]
we need to prove
\begin{eqnarray*}
&&-(-1)^{|a'||b|+|a||a'|+|b||b'|}m^A_{a'}h^A_a\otimes h^B_{b'}m^B_b-(-1)^{|a'||b|+|a||a'|+|b||b'|}h^A_{a'}m^A_a\otimes m^B_{b'}h^B_b\\
&=&(-1)^{|a'||b|+|b||b'|}m^A_{\{a,a'\}}\otimes m^B_{b'}h^B_b-(-1)^{|a'||b|+|b||b'|}m^A_ah^A_{a'}\otimes m^B_{b'}h^B_b\\
&&-(-1)^{|a'||b|+|a||a'|}m^A_{a'}h^A_a\otimes m^B_bh^B_{b'}+(-1)^{|a'||b|+|a||a'|}m^A_{a'}h^A_a\otimes m^B_{\{b,b'\}},
\end{eqnarray*}
which follows from
\begin{eqnarray*}
&&(-1)^{|a'||b|+|b||b'|}m^A_{\{a,a'\}}\otimes m^B_{b'}h^B_b-(-1)^{|a'||b|+|b||b'|}m^A_ah^A_{a'}\otimes m^B_{b'}h^B_b\\
&=&-(-1)^{|a'||b|+|a||a'|+|b||b'|}(-(-1)^{|a||a'|}m^A_{\{a,a'\}}+(-1)^{|a||a'|}m^A_ah^A_{a'})\otimes m^B_{b'}h^B_b\\
&=&-(-1)^{|a'||b|+|a||a'|+|b||b'|}(m^A_{\{a',a\}}+(-1)^{|a||a'|}m^A_ah^A_{a'})\otimes m^B_{b'}h^B_b\\
&=&-(-1)^{|a'||b|+|a||a'|+|b||b'|}h^A_{a'}m^A_a\otimes m^B_{b'}h^B_b\end{eqnarray*}
and
\begin{eqnarray*}
&&-(-1)^{|a'||b|+|a||a'|}m^A_{a'}h^A_a\otimes m^B_bh^B_{b'}+(-1)^{|a'||b|+|a||a'|}m^A_{a'}h^A_a\otimes m^B_{\{b,b'\}}\\
&=&-(-1)^{|a'||b|+|a||a'|+|b||b'|}m^A_{a'}h^A_a\otimes ((-1)^{|b||b'|}m^B_bh^B_{b'}-(-1)^{|b||b'|}m^B_{\{b,b'\}})\\
&=&-(-1)^{|a'||b|+|a||a'|+|b||b'|}m^A_{a'}h^A_a\otimes ((-1)^{|b||b'|}m^B_bh^B_{b'}+m^B_{\{b',b\}})\\
&=&-(-1)^{|a'||b|+|a||a'|+|b||b'|}m^A_{a'}h^A_a\otimes h^B_{b'}m^B_b.\end{eqnarray*}

The graded Lie algebra map $h$ is compatible with the differential since
\begin{eqnarray*}
&&hd_{A\otimes B}(a\otimes b)\\
&=&(h^A\otimes m^B+m^A\otimes h^B)(d_A(a)\otimes b+(-1)^{|a|}a\otimes d_B(b))\\
&=&h^A_{d_A(a)}\otimes m^B_b+(-1)^{|a|}h^A_a\otimes m^B_{d_B(b)}+m^A_{d_A(a)}\otimes h^B_b+(-1)^{|a|}m^A_a\otimes h^B_{d_B(b)}\\
&=&\partial_{A^e}(h^A_a)\otimes m^B_b+(-1)^{|a|}h^A_a\otimes\partial_{B^e}(m^B_b)+\partial_{A^e}(m^A_a)\otimes h^B_b+(-1)^{|a|}m^A_a\otimes\partial_{B^e}(h^B_b)\\
&=&\partial_{A^e\otimes B^e}(h^A_a\otimes m^B_b)+\partial_{A^e\otimes B^e}(m^A_a\otimes h^B_b)\\
&=&\partial_{A^e\otimes B^e}(h^A\otimes m^B+m^A\otimes h^B)(a\otimes b)\\
&=&\partial_{A^e\otimes B^e}h(a\otimes b).\end{eqnarray*}
\end{proof}
\begin{lemma}\label{ot3}
Retain the notations of Lemma \ref{ot1} and \ref{ot2}. For any homogeneous elements $a\otimes b, a'\otimes b'\in A\otimes B$, we have
\[
m(\{a\otimes b,a'\otimes b'\}_{A\otimes B})=h(a\otimes b)m(a'\otimes b')-(-1)^{(|a|+|b|)(|a'|+|b'|)}m(a'\otimes b')h(a\otimes b)
\]
and
\[
h((a\otimes b)(a'\otimes b'))=m(a\otimes b)h(a'\otimes b')+(-1)^{(|a|+|b|)(|a'|+|b'|)}m(a'\otimes b')h(a\otimes b).
\]
\end{lemma}
\begin{proof}
The first equation follows from
\begin{eqnarray*}
&&m(\{a\otimes b,a'\otimes b'\}_{A\otimes B})\\
&=&(m^A\otimes m^B)(\{a\otimes b,a'\otimes b'\}_{A\otimes B})\\
&=&(-1)^{|a'||b|}(m^A_{aa'}\otimes m^B_{\{b,b'\}_B}+m^A_{\{a,a'\}_A}\otimes m^B_{bb'})\\
&=&(-1)^{|a'||b|}(m^A_am^A_{a'}\otimes (h^B_bm^B_{b'}-(-1)^{|b||b'|}m^B_{b'}h^B_b)+(h^A_am^A_{a'}-(-1)^{|a||a'|}m^A_{a'}h^A_a)\otimes m^B_bm^B_{b'})\\
&=&(-1)^{|a'||b|}m^A_am^A_{a'}\otimes h^B_bm^B_{b'}-(-1)^{|a'||b|+|b||b'|}m^A_am^A_{a'}\otimes m^B_{b'}h^B_b+(-1)^{|a'||b|}h^A_am^A_{a'}\otimes m^B_bm^B_{b'}\\
&&-(-1)^{|a'||b|+|a||a'|}m^A_{a'}h^A_a\otimes m^B_bm^B_{b'}
\end{eqnarray*}
and
\begin{eqnarray*}
&&h(a\otimes b)m(a'\otimes b')-(-1)^{(|a|+|b|)(|a'|+|b'|)}m(a'\otimes b')h(a\otimes b)\\
&=&(m^A_a\otimes h_b^B+h^A_a\otimes m^B_b)(m^A_{a'}\otimes m^B_{b'})-(-1)^{(|a|+|b|)(|a'|+|b'|)}(m^A_{a'}\otimes m^B_{b'})(h^A_a\otimes m^B_b+m^A_a\otimes h^B_b)\\
&=&(-1)^{|a'||b|}m^A_am^A_{a'}\otimes h^B_bm^B_{b'}-(-1)^{|a'||b|+|b||b'|}m^A_am^A_{a'}\otimes m^B_{b'}h^B_b+(-1)^{|a'||b|}h^A_am^A_{a'}\otimes m^B_bm^B_{b'}\\
&&-(-1)^{|a'||b|+|a||a'|}m^A_{a'}h^A_a\otimes m^B_bm^B_{b'}.\end{eqnarray*}

The second equation follows from
\begin{eqnarray*}
&&h((a\otimes b)(a'\otimes b'))\\
&=&(-1)^{|a'||b|}(h^A\otimes m^B+m^A\otimes h^B)(aa'\otimes bb')\\
&=&(-1)^{|a'||b|}(h^A_{aa'}\otimes m^B_{bb'}+m^A_{aa'}\otimes h^B_{bb'})\\
&=&(-1)^{|a'||b|}((m^A_ah^A_{a'}+(-1)^{|a||a'|}m^A_{a'}h^A_a)\otimes m^B_bm^B_{b'}+m^A_am^A_{a'}\otimes (m^B_bh^B_{b'}+(-1)^{|b||b'|}m^B_{b'}h^B_b))\\
&=&(-1)^{|a'||b|}m^A_ah^A_{a'}\otimes m^B_bm^B_{b'}+(-1)^{|a'||b|+|a||a'|}m^A_{a'}h^A_a\otimes m^B_bm^B_{b'}+(-1)^{|a'||b|}m^A_am^A_{a'}\otimes m^B_bh^B_{b'}\\
&&+(-1)^{|a'||b|+|b||b'|}m^A_am^A_{a'}\otimes m^B_{b'}h^B_b
\end{eqnarray*}
and

\begin{eqnarray*}
&&m(a\otimes b)h(a'\otimes b')+(-1)^{(|a|+|b|)(|a'|+|b'|)}m(a'\otimes b')h(a\otimes b)\\
&=&(m^A_a\otimes m^B_b)(m^A_{a'}\otimes h^B_{b'}+h^A_{a'}\otimes m^B_{b'})+(-1)^{(|a|+|b|)(|a'|+|b'|)}(m^A_{a'}\otimes m^B_{b'})(m^A_{a}\otimes h^B_{b}+h^A_{a}\otimes m^B_{b})\\
&=&(-1)^{|a'||b|}m^A_am^A_{a'}\otimes m^B_bh^B_{b'}+(-1)^{|a'||b|}m^A_ah^A_{a'}\otimes m^B_bm^B_{b'}+(-1)^{|a'||b|+|a||a'|+|b||b'|}m^A_{a'}h^A_a\otimes m^B_{b'}m^B_{b}\\
&&+(-1)^{|a'||b|+|a||a'|+|b||b'|}m^A_{a'}m^A_{a}\otimes m^B_{b'}h^B_b\\
&=&(-1)^{|a'||b|}m^A_am^A_{a'}\otimes m^B_bh^B_{b'}+(-1)^{|a'||b|}m^A_ah^A_{a'}\otimes m^B_bm^B_{b'}+(-1)^{|a'||b|+|a||a'|}m^A_{a'}h^A_a\otimes m^B_bm^B_{b'}\\
&&+(-1)^{|a'||b|+|b||b'|}m^A_am^A_{a'}\otimes m^B_{b'}h^B_b.\end{eqnarray*}
\end{proof}

\begin{lemma}\label{ot4}
Let $A,B\in\mathbf{DGPA}$ and $(D,\delta)\in\mathbf{DGA}$. Let
$$p: (A\otimes B, \partial_{A\otimes B})\to (D,\delta)$$ be a differential graded algebra map and
$$q: (A\otimes B, \{,\}_{A\otimes B},\partial_{A\otimes B})\to (D,[,],\delta)$$ be a differential graded Lie algebra map such that
\[
p(\{a\otimes b,a'\otimes b'\}_{A\otimes B})=q(a\otimes b)p(a'\otimes b')-(-1)^{(|a|+|b|)(|a'|+|b'|)}p(a'\otimes b')q(a\otimes b)
\]
and
\[
q((a\otimes b)(a'\otimes b'))=p(a\otimes b)q(a'\otimes b')+(-1)^{(|a|+|b|)(|a'|+|b'|)}p(a'\otimes b')q(a\otimes b)
\]
for all homogeneous elements $a\otimes b, a'\otimes b'\in A\otimes B$. Let $i_A: A\to A\otimes B$ send $a$ to $a\otimes 1$ and $i_B: B\to A\otimes B$ send $a$ to $1\otimes b$. Denote
\[
p_A:=pi_A,\;q_A:=qi_A,\;p_B:=pi_B,\;q_B:=qi_B.
\]
Then
\begin{enumerate}
  \item $p_A$ is a differential graded algebra map and $q_A$ is a differential graded Lie algebra map such that
 \[
 p_A(\{a,a'\})=q_A(a)p_A(a')-(-1)^{|a||a'|}p_A(a')q_A(a)\] and
\[q_A(aa')=p_A(a)q_A(a')+(-1)^{|a||a'|}p_A(a')q_A(a)
 \] for any homogeneous elements $a,a'\in A$.
 \item $p_B$ is a differential graded algebra map and $q_B$ is a differential graded Lie algebra map such that
 \[
 p_B(\{b,b'\})=q_B(b)p_B(b')-(-1)^{|b||b'|}p_B(b')q_B(b)\] and
\[q_B(bb')=p_B(b)q_B(b')+(-1)^{|b||b'|}p_B(b')q_B(b)
 \] for any homogeneous elements $b,b'\in B$.
\end{enumerate}
\end{lemma}
\begin{proof}
It suffices to prove (1) since (2) can be proved similarly. Let $a,a'\in A$ be any homogeneous elements. Note that $p_A$ is a differential graded algebra map follows from
\[
p_A(aa')=p(aa'\otimes 1)=p((a\otimes 1)(a'\otimes 1))=p(a\otimes 1)p(a'\otimes 1)=p_A(a)p_A(a')
\]and
\[
p_Ad_A(a)=p(d_A(a)\otimes 1)=pd_{A\otimes B}(a\otimes 1)=\delta p(a\otimes 1)=\delta p_A(a).
\]

$q_A$ is a differential graded Lie algebra map follows from
\[
q_A(\{a,a'\}_A)=q(\{a,a'\}_A\otimes 1)=q(\{a,a'\}_A\otimes 1+aa'\otimes \{1,1\}_B)=q(\{a\otimes 1,a'\otimes 1\}_{A\otimes B})=[q_A(a),q_A(a')]\]and
\[
q_Ad_A(a)=q(d_A(a)\otimes 1+(-1)^{|a|}a\otimes d_B(1))=qd_{A\otimes B}(a\otimes 1)=\delta q(a\otimes 1)=\delta q_A(a).
\]

Now note that

\begin{eqnarray*}
p_A(\{a,a'\}_A)&=&p(\{a,a'\}_A\otimes 1)\\
&=&p(\{a\otimes 1,a'\otimes 1\}_{A\otimes B})\\
&=&q(a\otimes 1)p(a'\otimes1)-(-1)^{|a||a'|}p(a'\otimes 1)q(a\otimes1)\\
&=&q_A(a)p_A(a')-(-1)^{|a||a'|}p_A(a')q_A(a)
\end{eqnarray*}
and
\begin{eqnarray*}
q_A(aa')&=&q(aa'\otimes 1)\\
&=&q((a\otimes 1)(a'\otimes 1))\\
&=&p(a\otimes 1)q(a'\otimes1)+(-1)^{|a||a'|}p(a'\otimes1)q(a\otimes1)\\
&=&p_A(a)q_A(a')+(-1)^{|a||a'|}p_A(a')q_A(a),
\end{eqnarray*}
which complete the proof.
\end{proof}

\begin{lemma}\label{ot5}
Retain the Notations of Lemma \ref{ot4}. We have the following statements:
\begin{enumerate}
  \item There is a unique differential graded algebra map $f_A:(A^e,\partial_{A^e})\to (D,\delta)$ such that the following diagram
  \[
\xymatrix{
A\ar[rr]^-{m^A}_-{h^A}\ar[dr]^-{p_A}_-{q_A} && A^{e}\ar@{-->}[dl]^-{\exists ! f_A}\\
& D&
}
\]``bi-commutes''.
  \item There is a unique differential graded algebra map $f_B:(B^e,\partial_{B^e})\to (D,\delta)$ such that the following diagram
  \[
\xymatrix{
B\ar[rr]^-{m^B}_-{h^B}\ar[dr]^-{p_B}_-{q_B} && B^{e}\ar@{-->}[dl]^-{\exists ! f_B}\\
& D&
}
\]``bi-commutes''.
\end{enumerate}
\end{lemma}
\begin{proof}
It suffices to prove (1) since (2) can be proved similarly.

Note that $A^e$ is the universal enveloping algebra of $A$, $m^A$ is a differential graded algebra map and $h^A$ is a differential graded Lie algebra map satisfying
\begin{align*}
h^A_{aa'}=m^A_ah^A_{a'}+(-1)^{|a||a'|}m^A_{a'}h^A_a,\\
m^A_{\{a,a'\}}=h^A_am^A_{a'}-(-1)^{|a||a'|}m^A_{a'}h^A_a,
\end{align*}
for any homogeneous elements $a,a'\in A$.
By Lemma \ref{ot4}, we have that $p_A$ is a differential graded algebra map and $q_A$ is a differential graded Lie algebra map such that
 \[
 p_A(\{a,a'\})=q_A(a)p_A(a')-(-1)^{|a||a'|}p_A(a')q_A(a)\] and
\[q_A(aa')=p_A(a)q_A(a')+(-1)^{|a||a'|}p_A(a')q_A(a)
 \] for any homogeneous elements $a,a'\in A$. Now by Definition \ref{univ}, we finish the proof of (1).
\end{proof}

\begin{lemma}\label{ot6}
Retain the notations of Lemma \ref{ot5}. For any homogeneous elements $x\in A^e$ and $y\in B^e$, we have
\[
f_A(x)f_B(y)=(-1)^{|x||y|}f_B(y)f_A(x).
\]
\end{lemma}
\begin{proof}
Note that $f_A$ and $f_B$ are graded algebra maps, we only need to check the equation on the homogeneous generators. Now the result follows from
\begin{eqnarray*}
f_A(m^A_a)f_B(m^B_b)&=&p_A(a)p_B(b)\\
&=&p(a\otimes 1)p(1\otimes b)\\
&=&p(a\otimes b)\\
&=&(-1)^{|a||b|}p((1\otimes b)(a\otimes1))\\
&=&(-1)^{|a||b|}p_B(b)p_A(a)\\
&=&(-1)^{|a||b|}f_B(m^B_b)f_A(m^A_a),
\end{eqnarray*}
 \begin{eqnarray*}
 f_A(m^A_a)f_B(h^B_b)&=&p_A(a)q_B(b)\\
 &=&p(a\otimes 1)q(1\otimes b)\\
 &=&(-1)^{|a||b|}(q(1\otimes b)p(a\otimes 1)-p(\{1\otimes b,a\otimes 1\}_{A\otimes B}))\\
 &=&(-1)^{|a||b|}q(1\otimes b)p(a\otimes 1)\\
 &=&(-1)^{|a||b|}q_B(b)p_A(a)\\
 &=&(-1)^{|a||b|}f_B(h^B_b)f_A(m^A_a),
 \end{eqnarray*}
 \begin{eqnarray*}
 f_A(h^A_a)f_B(m^B_b)&=&q_A(a)p_B(b)\\
 &=&q(a\otimes 1)p(1\otimes b)\\
 &=&p(\{a\otimes1,1\otimes b\}_{A\otimes B})+(-1)^{|a||b|}p(1\otimes b)q(a\otimes1)\\
 &=&(-1)^{|a||b|}p(1\otimes b)q(a\otimes1)\\
 &=&(-1)^{|a||b|}p_B(b)q_A(a)\\
 &=&(-1)^{|a||b|}f_B(m^B_b)f_A(h^A_a)
\end{eqnarray*}
 and
\begin{eqnarray*}
 f_A(h^A_a)f_B(h^B_b)&=&q_A(a)q_B(b)\\
 &=&q(a\otimes1)q(1\otimes b)\\
 &=&q(\{a\otimes1,1\otimes b\}_{A\otimes B})+(-1)^{|a||b|}q(1\otimes b)q(a\otimes 1)\\
 &=&(-1)^{|a||b|}q(1\otimes b)q(a\otimes 1)\\
 &=&(-1)^{|a||b|}q_B(b)q_A(a)\\
 &=&(-1)^{|a||b|}f_B(h^B_b)f_A(h^A_a)
\end{eqnarray*}
for any homogeneous elements $a\in A$ and $b\in B$.
 \end{proof}

Now we are ready to state and prove our main result.
\begin{theorem}\label{thotimes}
Let $A,B\in\mathbf{ DGPA}$. Then $(A\otimes B)^e\cong A^e\otimes B^e$.
\end{theorem}
\begin{proof}
Define a $\Bbbk$-linear map
\[
f:=f_A\otimes f_B: A^e\otimes B^e\to D\;\;{\rm via}\;\;f(x\otimes y)=f_A(x)f_B(y)
\] for any homogeneous element $x\otimes y\in A^e\otimes B^e$.

Now we claim that $f$ is a differential graded algebra map. In fact, by Lemma \ref{ot6}, for any homogeneous elements $x\otimes y, x'\otimes y'\in A^e\otimes B^e$, we have
\begin{eqnarray*}
f((x\otimes y)(x'\otimes y'))&=&(-1)^{|x'||y|}f(xx'\otimes yy')\\
&=&(-1)^{|x'||y|}f_A(xx')f_B(yy')\\
&=&(-1)^{|x'||y|}f_A(x)f_A(x')f_B(y)f_B(y')\\
&=&f_A(x)f_B(y)f_A(x')f_B(y')\\
&=&f(x\otimes y)f(x'\otimes y')
\end{eqnarray*}
and
\begin{eqnarray*}
f\partial_{A^e\otimes B^e}(x\otimes y)&=&(f_A\otimes f_B)(\partial_{A^e}(x)\otimes y+(-1)^{|x|}x\otimes \partial_{B^e}(y))\\
&=&f_A\partial_{A^e}(x)f_B(y)+(-1)^{|x|}f_A(x)f_B\partial_{B^e}(y)\\
&=&\delta f_A(x)f_B(y)+(-1)^{|x|}f_A(x)\delta f_B(y)\\
&=&\delta(f_A\otimes f_B)(x\otimes y)\\
&=&\delta f(x\otimes y).
\end{eqnarray*}
Therefore, by Lemmas \ref{ot4} and \ref{ot5}, we have the following ``bi-commutative'' diagram:

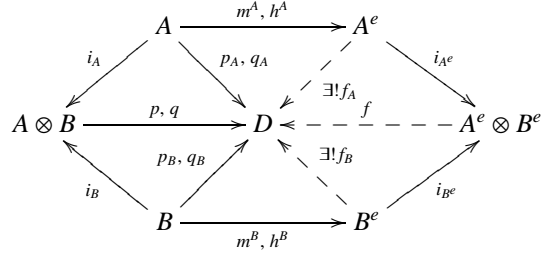
\[
\xymatrix{
& A\ar[dl]_-{i_A}\ar[dr]^-{p_A,\;q_A}\ar[rr]^-{m^A,\; h^A} & & A^e\ar@{-->}[dl]^-{\exists !f_{A}}\ar[dr]^-{i_{A^e}}&\\
A\otimes B\ar[rr]^-{p,\;q}&  &D & & A^e\otimes B^e\ar@{-->}[ll]_-{f}\\
& B\ar[ur]^-{p_B,\;q_B}\ar[ul]^-{i_B}\ar[rr]_-{m^B, \;h^B}& & B^e\ar@{-->}[ul]_-{\exists !f_{B}}\ar[ur]_-{i_{B^e}}&\\
}
\]
Here $(D,\delta)$ is any differential graded algebra with a differential graded algebra map $p: A\otimes B\to D$ and a differential graded Lie algebra map $q:A\otimes B\to D$ such that
\[
p(\{a\otimes b,a'\otimes b'\}_{A\otimes B})=q(a\otimes b)p(a'\otimes b')-(-1)^{(|a|+|b|)(|a'|+|b'|)}p(a'\otimes b')q(a\otimes b)
\]
and
\[
q((a\otimes b)(a'\otimes b'))=p(a\otimes b)q(a'\otimes b')+(-1)^{(|a|+|b|)(|a'|+|b'|)}p(a'\otimes b')q(a\otimes b)
\]
for all homogeneous elements $a\otimes b, a'\otimes b'\in A\otimes B$. By Lemmas \ref{ot1}, \ref{ot2} and \ref{ot3}, we have that $$m:=m^A\otimes m^B: (A\otimes B, d_{A\otimes B}\to (A^e\otimes B^e,\partial_{A^e\otimes B^e})$$ is a differential graded algebra map and \[h:=m^A\otimes h^B+h^A\otimes m^B: (A\otimes B, \{,\}_{A\otimes B}, d_{A\otimes B})\to (A^e\otimes B^e, [,], \partial_{A^e\otimes B^e})\]
is a differential graded Lie algebra map satisfying
\[
m(\{a\otimes b,a'\otimes b'\}_{A\otimes B})=h(a\otimes b)m(a'\otimes b')-(-1)^{(|a|+|b|)(|a'|+|b'|)}m(a'\otimes b')h(a\otimes b)
\]
and
\[
h((a\otimes b)(a'\otimes b'))=m(a\otimes b)h(a'\otimes b')+(-1)^{(|a|+|b|)(|a'|+|b'|)}m(a'\otimes b')h(a\otimes b)
\]
for any homogeneous elements $a\otimes b, a'\otimes b'\in A\otimes B$.

Now we claim that we have the following ``bi-commutative'' diagram
 \[
\xymatrix{
A\otimes B\ar[rr]^-{m}_-{h}\ar[dr]^-{p}_-{q} && A^e\otimes B^{e}\ar[dl]^-{ f}\\
& D&
}
\]
Indeed, for any homogeneous element $a\otimes b\in A\otimes B$, $fm=p$ follows from
\begin{eqnarray*}
fm(a\otimes b)&=&f(m^A_a\otimes m^B_b)\\
&=&f_Am^A_af_Bm^B_b\\
&=&p_A(a)p_B(b)\\
&=&pi_A(a)pi_B(b)\\
&=&p(a\otimes1)p(1\otimes b)\\
&=&p(a\otimes b),
\end{eqnarray*}
and $fh=q$ follows from
\begin{eqnarray*}
fh(a\otimes b)&=&f(m^A_a\otimes h^B_b+h^A_a\otimes m^B_b)\\
&=&f_Am^A_af_Bh^B_b+f_Ah^A_af_Bm^B_b\\
&=&p_A(a)q_B(b)+q_A(a)p_B(b)\\
&=&pi_A(a)qi_B(b)+qi_A(a)pi_B(b)\\
&=&p(a\otimes 1)q(1\otimes b)+q(a\otimes1)p(1\otimes b)\\
&=&p(a\otimes 1)q(1\otimes b)+p(\{a\otimes 1,1\otimes b\}_{A\otimes B})+(-1)^{|a||b|}p(1\otimes b)q(a\otimes 1)\\
&=&p(a\otimes 1)q(1\otimes b)+(-1)^{|a||b|}p(1\otimes b)q(a\otimes1)\\
&=&q(a\otimes b).
\end{eqnarray*}
By the universal property of $A^e$, $B^e$ and the construction of $f$, we know that such $f$ is also unique. Now by Definition \ref{univ}, we know that
$$(A\otimes B)^e\cong A^e\otimes B^e,$$
which completes the proof.
\end{proof}

As a direct application of Theorem \ref{thotimes}, we have
\begin{corollary}\label{mono}
The functor ``$e$'' in Corollary \ref{e} is a monoidal functor.
\end{corollary}
\begin{proof}
By Proposition \ref{tensor} and Theorem \ref{thotimes}, it is enough to prove $\Bbbk^e=\Bbbk$. By the construction of $\Bbbk^e$, we have
\[h_1=h_{1\times1}=m_1h_1+(-1)^{|1||1|}m_1h_1=2h_1,\]which implies that $h_1=0$. Hence $\Bbbk^e=\Bbbk$.
\end{proof}

Moreover, we can build a link between universal enveloping algebra of a differential graded Poisson algebra and the classic universal enveloping algebra of an algebra:
\begin{corollary}\label{coop}
Let $A$ be a differential graded Poisson algebra. Then $(A^{\mathcal{E}})^e\cong(A^e)^{\mathcal{E}}$.
\end{corollary}
\begin{proof}
By Theorems \ref{thop} and \ref{thotimes}, we have
\[
(A^{\mathcal{E}})^e=(A\otimes A^{op})^e\cong A^e\otimes (A^{op})^e\cong A^e\otimes (A^e)^{op}=(A^e)^{\mathcal{E}},\]
as required.
\end{proof}

We conclude this paper with the following observation:
\begin{remark}\label{4.13}
By Theorem \ref{equivalence}, we know that the differential graded Poisson module category $\mathbf{DGP}(A)$ is isomorphic to the differential graded module category $\mathbf{DG}(A^e)$, which can be localized at quasi-isomorphisms such that it becomes a triangulated category. Hence, we can study the (co)homology theories in $\mathbf{DGP}(A)$. In particular, we can define the Hochschild (co)homology for $A$ as a differential graded Poisson $A$-bimodule, or equivalently, a differential graded Poisson $A\otimes A^{op}$-module. From the previous results, we know it is the same as the Hochschild (co)homology of $A$ as a differential graded $(A\otimes A^{op})^e=A^e\otimes (A^{op})^e=A^e\otimes (A^e)^{op}$ module, or differential graded $A^e$-bimodule.
\end{remark}

\end{document}